\documentclass[12pt,reqno,a4paper]{amsart}
  \usepackage[bookmarksnumbered=true]{hyperref}

\usepackage{latexsym, enumerate}
\usepackage{amssymb,amsmath,mathtools,stmaryrd,dsfont}

\usepackage{mathtools}

\usepackage{tikz}
\usetikzlibrary{shapes.geometric,arrows}
\usepackage{multicol}

\tikzstyle{box} = [rectangle,text centered, draw=black]
\tikzstyle{arrow} = [thick, ->, >=stealth]

\usepackage{color}

\usepackage{multirow}

\usepackage{ulem}
\usepackage{changes}

\mathtoolsset{showonlyrefs}

\usepackage{comment}

\newcommand{\rd}{{\rm d}}
\newcommand{\e}{{\rm e}}
\newcommand{\Ran}{\mathop{\rm Ran}}

\newcommand{\R}{{\mathbb R}}
\newcommand{\C}{{\mathbb C}}
\newcommand{\Z}{{\mathbb Z}}
\newcommand\eps{\epsilon}

\newcommand\re{\mathrm{Re}\,}
\newcommand\im{\mathrm{Im}\,}
\newcommand\I{\mathrm{i}}

\DeclareMathOperator{\tr}{Tr}
\DeclareMathOperator{\supp}{supp}

\newcommand{\g}{{\rm g}}
\DeclareMathOperator{\spec}{spec}
\DeclareMathOperator{\dom}{Dom}

\newcommand{\abs}[1]{\left\lvert #1 \right\rvert}
\newcommand{\norm}[1]{\left\lVert #1 \right\rVert}

\newcommand{\normLp}[3]{\lVert #1 \rVert_{L^{#2}#3}} 
\newcommand{\schatten}{\mathfrak{S}} 
\newcommand{\normSch}[3]{\left\lVert #1 \right\rVert_{\mathfrak{S}^{#2}#3}} 

\newcommand\Hc{\mathcal{H}}

\newtheorem{theorem}{Theorem}[section]

\newtheorem{definition}[theorem]{Definition}
\newtheorem{proposition}[theorem]{Proposition}

\newtheorem{lemma}[theorem]{Lemma}
\newtheorem{remark}[theorem]{Remark}

\begin{document}

\title[{ORTHONORMAL SPECTRAL CLUSTER BOUNDS FOR NONSMOOTH METRICS}]{SPECTRAL CLUSTER BOUNDS FOR ORTHONORMAL FUNCTIONS ON MANIFOLDS WITH NONSMOOTH METRICS}

\author{Jean-Claude Cuenin}
\address[Jean-Claude Cuenin]{Department of Mathematical Sciences, Loughborough University, Loughborough, Leicestershire, LE11 3TU United Kingdom}
\email{J.Cuenin@lboro.ac.uk}

\author{Ngoc Nhi Nguyen}
\address[Ngoc Nhi Nguyen]{Università degli Studi di Milano, Dipartimento di Matematica, Via Cesare Saldini 50, 20133 Milano, Italy}
\email{ngoc.nguyen@unimi.it}

\author{Xiaoyan Su}
\address[Xiaoyan Su]{Department of Mathematical Sciences, Loughborough University, Loughborough, Leicestershire, LE11 3TU United Kingdom}
\email{X.Su2@lboro.ac.uk}

\date{May 2025}

\thanks{\copyright\, 2025 by the authors. This paper may be reproduced, in its entirety, for non-commercial purposes.}

\begin{abstract}
We establish $L^q$ spectral cluster bounds for families of orthonormal functions associated to the Laplace--Beltrami operator on a compact Riemannian manifold. The metric is only assumed to be of class $C^s$, where $0\leq s\leq 2$.
\end{abstract}

\maketitle

\section{Introduction and main results}
\subsection{Smooth metrics}
Let $(M,\g)$ be a smooth compact Riemannian manifold without boundary and of dimension $n\geq 2$. We denote by $\Delta_{\g}$ the (negative) Laplace–Beltrami operator on $M$, whose action on smooth functions, in local coordinates, is given by
\begin{align}
\Delta_{\g}u=|\g|^{-1/2}\partial_i(|\g|^{1/2}\g^{ij}\partial_j u),
\end{align}
where $|\g|=\det (\g_{ij})$ and the summation convention is in effect.
We consider $-\Delta_{\g}$ as a 
non-negative, unbounded self-adjoint operator in $L^2(M)$, defined with respect to the Riemannian measure $|\g|^{1/2}\rd x$. For $\lambda\geq 1$, the spectral projection
\begin{align*}
\Pi_{\lambda}:=\mathds{1}(\sqrt{-\Delta_{\g}}\in [\lambda,\lambda+1]),
\end{align*}
defined by the functional calculus, is the orthogonal projection onto the subspace of functions in $L^2(M)$ that are spectrally localized to frequencies in the unit length window $[\lambda,\lambda+1]$. 
We also define the spectral cluster 
\begin{align}
    E_{\lambda}:=\Ran \Pi_{\lambda}\subset L^2(M).
\end{align}
If the metric $\g$ is smooth,
a seminal result of Sogge \cite{Sogge-1988-concerning} states that
\begin{align}\label{eq:Sogge}
    \|u\|_{L^q(M)}\leq C \lambda^{\delta(q)} \|u\|_{L^2(M)},\quad u\in E_{\lambda},
\end{align}
where $2\leq q\leq \infty$ and
\begin{equation}\label{eq-def:delta_sogge}
    \delta(q)=\begin{cases}
        \frac{n-1}{2}\left(\frac12-\frac1q\right) & \text{if}\ 2\leq q\leq q_n \, ,\\
        n\left(\frac12-\frac1q\right)-\frac12 & \text{if}\ q_n\leq q\leq \infty ,
       \end{cases}
       \quad q_n=\frac{2(n+1)}{n-1}.
\end{equation}
More recently, also for smooth metrics, Frank and Sabin \cite[Theorem 3]{FrankSabin-2017clust} proved the following generalization of~\eqref{eq:Sogge} to families of orthonormal functions $(u_j)_{j\in J}\subset E_{\lambda}$,
\begin{align}\label{Frank-Sabin orthoclusters}
\bigg\|\sum_{j\in J}\nu_j|u_j|^2\bigg\|_{L^{q/2}(M)}\leq C \lambda^{2\delta(q)}\bigg(\sum_{j\in J}|\nu_j|^{\alpha(q)}\bigg)^{1/\alpha(q)},
\end{align}
where $(\nu_j)_{j\in J}\subset \C$ and
\begin{equation}\label{eq-def:alpha}
    \alpha(q) :=\begin{cases}
    \frac{2q}{q+2} & \text{if}\ 2\le q\le q_n ,\\
    \frac{q(n-1)}{2n} & \text{if}\ q_n\le q\le \infty.
       \end{cases}
\end{equation}
Here and in the following, $J$ denotes any countable index set. In comparison, using \eqref{eq:Sogge} and the triangle inequality, one only obtains \eqref{Frank-Sabin orthoclusters} with $\alpha=1$ (without the orthogonality assumption). The crucial feature of \eqref{Frank-Sabin orthoclusters}--\eqref{eq-def:alpha} is that $\alpha>1$.
When $\#J=1$, \eqref{Frank-Sabin orthoclusters} coincides with Sogge's bound \eqref{eq:Sogge}. On the other extreme, when $\#J=\dim E_{\lambda}$, then \eqref{Frank-Sabin orthoclusters}, with $\nu_j=1$ for all $j\in J$, yields
\begin{align}\label{extreme case all functions in Frank-Sabin}
\bigg\|\sum_{j\in J}|u_j|^2\bigg\|_{L^{q/2}(M)}\leq C\lambda^{n-1},\quad 2\leq q\leq \infty,
\end{align}
due to the fact that $\dim E_{\lambda}\approx \lambda^{n-1}$, by the sharp Weyl law (see, e.g., \cite[Chapter 4]{MR3645429}). The sum in \eqref{extreme case all functions in Frank-Sabin} is independent of the choice of basis $(u_j)_{j\in J}$ of $E_{\lambda}$ and coincides with the spectral function corresponding to the interval $[\lambda,\lambda+1]$. The inequality \eqref{Frank-Sabin orthoclusters} can be seen as an optimal interpolation between the two extreme cases \eqref{eq:Sogge} and \eqref{extreme case all functions in Frank-Sabin}.

The insight that orthonormality improves the dependence on the number of functions, compared to merely applying the triangle inequality, originally stems from the celebrated work of Lieb and  Thirring \cite{LiebThirring1975,LiebThirring1976}, see also Lieb \cite{MR701053}, in the context of the Sobolev inequalities. More recently, this idea was extended to  Strichartz inequalities in, e.g., \cite{Frank-Lewin-Lieb-Seiringer-2014,FrankSabin-2017rest,MR3250975,MR3985036,MR4068269,MR4202486,MR4725147}. Spectral cluster estimates for semiclassical Schr\"odinger operators with smooth confining potentials were established by the second author in \cite[Theorem 7.2]{Nguyen-2022}, generalizing the single function bounds of Koch, Tataru, and Zworski \cite{Koch-Tataru-Zworski-2007}.

\subsection{Nonsmooth metrics}
An interesting question is whether the bounds \eqref{eq:Sogge} or \eqref{Frank-Sabin orthoclusters} hold for metrics of limited regularity. This is especially relevant for applications to nonlinear equations, where the metric evolves in time.

A breakthrough was achieved by Smith in \cite{Smith-2006-C11}, where he proved that~\eqref{eq:Sogge} remains valid for 
 $C^{1,1}$  metrics. This level of differentiability ensures uniqueness of the geodesic flow. More recently, Chen and Smith \cite{MR3900030} established the same bounds under the weaker assumption that the sectional curvatures of the manifold are uniformly bounded.

Counterexamples by Smith and Sogge \cite{MR1306017} as well as Smith and Tataru \cite{MR1909638} show that $C^{1,1}$ is indeed best possible in the $C^s$ scale. More precisely, they
constructed examples of $C^s$ metrics, for each $0<s<2$ (Lipschitz in case
$s = 1$), for which there are functions $u\in E_{\lambda}$ such that for all $q\geq 2$
\begin{align}\label{counterexamples Smith-Sogge and Smith-Tataru}
    \frac{\|u\|_{L^q(M)}}{\|u\|_{L^2(M)}}\geq c\lambda^{\frac{n-1}{2}(\frac{1}{2}-\frac{1}{q})(1+\frac{2-s}{2+s})},
\end{align}
where $c>0$. 
Thus, these examples exhibit multiplicative losses of $\lambda^{\delta(q)\frac{2-s}{2+s}}$, compared to the smooth case, for $2<q<2(n+2s^{-1})/(n-1)$; for later reference, we note that in the case $n=2$, $s=1$, this range is $2<q<8$.

In \cite{Smith-2006sharp}, Smith proved the upper bounds corresponding to \eqref{counterexamples Smith-Sogge and Smith-Tataru}, i.e.
\begin{align}\label{Smith upper bound on small balls}
    \|u\|_{L^q(M)}\leq C\lambda^{\frac{n-1}{2}(\frac{1}{2}-\frac{1}{q})(1+\frac{2-s}{2+s})},\quad u\in E_{\lambda},
\end{align}
for $1\leq s<2$ (Lipschitz if $s=1$) in the range $2\leq q\leq q_n$.
He also proved no loss estimates in the range $q_n\leq q\leq \infty$ on cubes of sidelength $R=\lambda^{-\frac{2-s}{2+s}}$. 
Since the constant in these estimates is uniform over all cubes, he also obtains global $L^{\infty}$ bounds for Lipschitz metrics that are as good as in the case of smooth metrics,
\begin{align}\label{Smith Linfty bounds}
    \|u\|_{L^{\infty}(M)}\leq C\lambda^{\frac{n-1}{2}},\quad u\in E_{\lambda}.
\end{align}
Interpolation between $q=q_n$ and $q=\infty$ yields bounds with exponents that are strictly larger than predicted by the
counterexamples in \cite{MR1306017}. Koch, Smith, and Tataru \cite{MR3282983} developed new techniques that combine energy flow estimates for Lipschitz coefficients with combinatorial arguments and proved no-loss estimates (up to a logarithmic loss) for Lipschitz ($s=1$) metrics in $n=2$ dimensions. This is best possible in view of the counterexamples in \cite{MR1306017}.

For H\"older metrics, that is $0< s< 1$, Koch, Smith, and Tataru \cite{KochSmithTataru-2007-holder} proved estimates with a further loss of $\lambda^{(1-s)(1/2-1/q)}$ in the range $2\leq q\leq q_n$. It is currently not known what the sharp bounds are for $q_n<q<\infty$.
Davies \cite{MR1037600} showed that for $s=0$, spectral cluster bounds are no better than what can be obtained by Sobolev embedding.

Related results include Blair's work \cite{MR2457396} on metrics with Sobolev regularity, as well as his work \cite{MR3148055} on restrictions of spectral clusters to submanifolds for low regularity metrics. We should also mention the work of Smith and Sogge \cite{MR2316270}, where sharp spectral cluster bounds on two-dimensional manifolds with boundaries are established. Although the metric is smooth to begin with, after doubling the manifold, one ends up with a manifold without boundary but with a Lipschitz metric.

The aim of this article is to extend (a subset of) the spectral cluster bounds for nonsmooth metrics, discussed above, to families of orthonormal functions. We will focus on the results of Smith \cite{Smith-2006-C11, Smith-2006sharp} (for the $C^{1,1}$ and $C^s$, $1\leq s<2$, case, respectively) and the result of Koch, Smith, and Tataru \cite{KochSmithTataru-2007-holder} for the H\"older case ($0<s<1$). We plan to revisit other results mentioned above at a later stage.

\subsection{Main results}\label{subsec:main-results}
The main results of this paper are the following generalizations of \cite{Smith-2006-C11,Smith-2006sharp,KochSmithTataru-2007-holder} to orthonormal systems.
For $s\in [0,2]$, we set 
\begin{align}
 \Pi_{\lambda}:=\mathds{1}(\sqrt{P}\in [\lambda,\lambda+\lambda^{(1-s)_+}]),\quad E_{\lambda}:=\Ran \Pi_{\lambda},
\end{align}
where $P=-\Delta_{\g}$ is defined by quadratic form methods, see Section \ref{subsec. quad. form} (due to the limited regularity of $\g$), and $x_+:=\max(x,0)$.

\begin{theorem}[The $C^{1,1}$ case]\label{thm:main-result_C2}
    Let $(M,\g)$ be a smooth compact boundaryless Riemannian manifold of dimension $n\geq 2$.
    Assume that the metric $\g$ is of class $C^{1,1}$.
    Then there exists $C>0$ such that for any $\lambda\geq 1$, any orthonormal system $(u_j)_{j\in J}\subset E_{\lambda}$ and any sequence $(\nu_j)_{j\in J}\subset \C$, the bound \eqref{Frank-Sabin orthoclusters} holds for all $2\leq q\leq\infty$, with $\delta(q)$ and $\alpha(q)$ as in \eqref{eq-def:delta_sogge} and \eqref{eq-def:alpha}, respectively.
\end{theorem}

\begin{theorem}[The case $s\in[1,2)$]\label{thm:main-result_Lip}
    Let $(M,\g)$ be a smooth compact boundaryless Riemannian manifold of dimension $n\geq 2$.
    Assume that the metric $\g$ is of class $C^s$ for $s\in[1,2)$ (Lipschitz in case $s=1$).
    Then there exists $C>0$ such that for any $\lambda\geq 1$, any orthonormal system $(u_j)_{j\in J}\subset E_{\lambda}$ and any sequence $(\nu_j)_{j\in J}\subset \C$, we have, for all $2\leq q\leq\infty$,
    \begin{align}\label{eq:spectral cluster bounds}
        \bigg\|\sum_{j\in J}\nu_j|u_j|^2\bigg\|_{L^{q/2}(M)}\leq C\lambda^{2\delta_s(q)}\bigg(\sum_{j\in J}|\nu_j|^{\alpha(q)}\bigg)^{1/\alpha(q)},
    \end{align}
     where $\alpha(q)$ is given by \eqref{eq-def:alpha}, and
    \begin{equation}\label{eq-def:delta-Lips}
        \delta_s(q):= \begin{cases}
             \frac{n-1}{2}\left(1+\frac{2-s}{2+s}\right)\left(\frac12-\frac1q\right)   &\text{if}\ 2\leq q\leq q_n,\\
            n\left(\frac12-\frac1q\right) -\frac 12+\frac{2-s}{2+s}\frac 1q&\text{if}\ q_n\leq q\leq \infty.
        \end{cases}
    \end{equation}    
    Moreover, for any cube $Q_R\subset M$ of sidelength $R =\lambda^{-\frac{2-s}{2+s}}$, we have
    \begin{align}\label{small_ball_estimates-Lip}
        \bigg\|\sum_{j\in J}\nu_j|u_j|^2\bigg\|_{L^{q/2}(Q_R)}\leq C\lambda^{2\delta(q)}\bigg(\sum_{j\in J}|\nu_j|^{\alpha(q)}\bigg)^{1/\alpha(q)},\quad q_n\leq q\leq \infty,
    \end{align}
     with $\delta(q)$ as in \eqref{eq-def:delta_sogge}.
\end{theorem}

\begin{theorem}[The case $s\in [0,1)$]\label{thm:main-result_Holder}
    Let $(M,\g)$ be a smooth compact boundaryless Riemannian manifold of dimension $n\geq 2$.
    Assume that the metric $\g$ is of class $C^s$ for $s\in[0,1)$.
    Then there exists $C>0$ such that for any $\lambda\geq 1$, any orthonormal system $(u_j)_{j\in J}\subset E_{\lambda}$ and any sequence $(\nu_j)_{j\in J}\subset \C$, we have, for all $2\leq q\leq\infty$,
    \begin{align}\label{eq:spectral_cluster_bounds-Holder}
        \bigg\|\sum_{j\in J}\nu_j|u_j|^2\bigg\|_{L^{q/2}(M)}\leq C\lambda^{2\delta_s(q)}\bigg(\sum_{j\in J}|\nu_j|^{\alpha(q)}\bigg)^{1/\alpha(q)},
    \end{align}    
    where $\alpha(q)$ is given by \eqref{eq-def:alpha}, and
    \begin{equation}\label{eq-def:delta-s-Holder}
        \delta_s(q):= \begin{cases}
            \left(\frac{n-1}{2}\left(1+\frac{2-s}{2+s}\right)+1-s\right)\left(\frac12-\frac1q\right)   &\text{if}\ 2\leq q\leq q_n,\\ 
             n\left(\frac12-\frac1q\right) -\frac 12+\frac{2-s}{2+s}\frac 1q+ (1-s)\left(\frac 12-\frac 1q\right) &\text{if}\ q_n\leq q\leq \infty.
        \end{cases}
    \end{equation}     
\end{theorem}

\begin{remark}
The exponents of $\lambda$ in Theorems \ref{thm:main-result_C2}--\ref{thm:main-result_Holder} are the same as in the one-function bounds \cite{Smith-2006-C11, Smith-2006sharp, KochSmithTataru-2007-holder}. The constants depend on the manifold $M$, the non-degeneracy of the metric $\g$ and its $C^s$ norm, but are independent of $\nu_j$, $u_j$ and $\lambda$.
\end{remark}

\subsection{Operator formulation}\label{subsec.reformulation-results}
As explained in \cite{FrankSabin-2017rest,Nguyen-2022}, bounds for systems of orthonormal functions like \eqref{Frank-Sabin orthoclusters} or \eqref{eq:spectral cluster bounds} can be expressed 
 more conveniently using operators. While the one-function estimate \eqref{eq:Sogge} is equivalent to
\begin{align}\label{eq:orthoclusters-1body}
    \norm{\Pi_\lambda}_{L^2(M)\to L^q(M)}\leq C\lambda^{\delta(q)},
\end{align}
the estimate \eqref{Frank-Sabin orthoclusters} for systems of orthonormal functions may be restated as follows (cf. \cite[Theorem 3]{FrankSabin-2017clust}). If $(u_j)_{j\in J}\subset E_{\lambda}$ is an orthonormal system and $(\nu_j)_{j\in J}\subset\C$ are coefficients, we consider the operator
\begin{align}\label{gamma}
	\gamma:=\sum_{j\in J}\nu_j|u_j\rangle\langle u_j|,
\end{align}
where we used Dirac's notation $(|u\rangle\langle v|)f:=\langle v,f\rangle u$, for all $u,v,f\in L^2(M)$. Then 
\begin{equation}\label{eq:orthoclusters-Mbody}
\normLp{\rho_\gamma}{q/2}{(M)}\leq C\lambda^{2\delta(q)}\normSch{\gamma}{\alpha(q)}{(L^2(M))},
\end{equation}
where $\rho_{\gamma}$ is called the density, and the norm on the right is a Schatten norm (see Section \ref{subsec:Schatten} for definitions), 
\begin{align*}
	\rho_\gamma(x):=\sum_{j\in J}\nu_j|u_j(x)|^2
	,\quad
	\normSch{\gamma}{\alpha(q)}{(L^2(M))}:=\bigg(\sum_{j\in J}\abs{\nu_j}^{\alpha(q)}\bigg)^{1/\alpha(q)}.
\end{align*}
By Cauchy--Schwarz, one may always assume without loss of generality that $\nu_j\geq 0$. Note that $\gamma=\Pi_\lambda\gamma \Pi_\lambda$, which implies that $\gamma$ has finite rank (because $\Pi_{\lambda}$ does).
We also remark that, in the case of smooth metrics, $\rho_{\gamma}$ is a smooth function (by elliptic regularity) and hence is pointwise defined. 
Using \cite[Prop. 1]{FrankSabin-2017rest}, one sees that \eqref{Frank-Sabin orthoclusters} is also equivalent to the estimate
\begin{align}\label{dual estimate in Frank-Sabin}
    \|W \Pi_\lambda W\|_{\mathfrak{S}^{\alpha(q)'}(L^2(M))} \leq C \lambda^{2\delta(q)} \|W\|^2_{L^{2(q/2)'}(M)}, 
\end{align}
that holds for any $W \in L^{2(q/2)'}(M)$, for some $C > 0$ independent of $W$ and  $\lambda > 1$. 
The operator formulation of \eqref{eq:spectral cluster bounds} is 
analogous, with $\delta(q)$ in \eqref{eq:orthoclusters-Mbody}, \eqref{dual estimate in Frank-Sabin} replaced by $\delta_s(q)$ and $\lambda\pm\I$ by $\lambda\pm\I\lambda^{(1-s)_+}$. 
In the following discussion, we will use \eqref{eq:orthoclusters-Mbody} as a reference, with the understanding that all the arguments work equally well if we replace $\delta(q)$ by $\delta_s(q)$.

Besides the dual formulation \eqref{dual estimate in Frank-Sabin}, another useful observation is that \eqref{eq:orthoclusters-Mbody} would follow from the \emph{quasimode bound}
\begin{align}\label{qm bound intro}
    \|\rho_{\gamma}\|_{L^{q/2}(M)}\lesssim\lambda^{2\delta(q)}\left(\normSch{\gamma}{\alpha(q)}{(L^2(M))}+\lambda^{-2}\normSch{(P-\lambda^2)\gamma(P-\lambda^2)}{\alpha(q)}{(L^2(M))}\right)
\end{align}
(see Lemma \ref{lemma gradient term}), where we assume that $\Ran(\gamma)\subset\dom(P)$. However, we do not assume that $\gamma=\Pi_{\lambda}\gamma \Pi_{\lambda}$ here (in contrast to \eqref{eq:orthoclusters-Mbody}). This means that $\gamma$ may have high-frequency components, and hence $q$ needs to be sufficiently small (dictated by Sobolev embedding or its analogue for operators, see Proposition \ref{prop:elliptic}, which gives $q\leq 2n/(n-2)$.) In any dimension, 
 $q=q_n$ is allowed, which will be sufficient (see Section \ref{subsect. remarks on the proof} below). The quasimode formulation is standard in the one-function case; in the many-function case, a similar version appeared in the work of the second author \cite{Nguyen-2022}. 
 The bound \eqref{qm bound intro} may be restated as a \textit{resolvent estimate},
 \begin{align}\label{resolvent bound intro}
     \big\|\rho\left((P-(\lambda+\I)^2)^{-1}\gamma (P-(\lambda+\I)^2)^{-1}\right)\big\|_{L^{q/2}(M)}\lesssim\lambda^{2\delta(q)-2}\normSch{\gamma}{\alpha(q)}{(L^2(M))}
 \end{align}
 (see Remark \ref{remark equivalence with A}) where we used the notation $\rho(A):=\rho_A$. By the duality principle of Frank and Sabin \cite[Lemma 3]{FrankSabin-2017rest} (see also Subsection \ref{subsection duality}), the latter is equivalent to the following bound,
 \begin{align}\label{dual resolvent bound intro}
\normSch{W(P-(\lambda+\I)^2)^{-1}(P-(\lambda-\I)^2)^{-1}\overline{W}}{\alpha(q)'}{(L^2(M))} \lesssim\lambda^{2\delta(q)-2}\|W\|^2_{L^{2(q/2)'}(M)}
 \end{align}
 for all $W \in L^{2(q/2)'}(M)$,
 which is essentially what we will prove.

\subsection{Remarks on the proof}\label{subsect. remarks on the proof}
We first note that the case $q=2$ in \eqref{eq:spectral cluster bounds} is trivial since $\|\rho_{\gamma}\|_{L^1}=\|\gamma\|_{\mathfrak{S}^1}$ and $\delta(2)=0$, $\alpha(2)=1$. As in \cite{FrankSabin-2017clust,Nguyen-2022}, the $L^\infty$ estimates in \eqref{eq:spectral cluster bounds} follow from the corresponding one-function bounds (see Proposition \ref{Prop.Linfty}). Thus, as in \cite{FrankSabin-2017clust,Smith-2006-C11,Smith-2006sharp}, we only need to consider the critical exponent $q=q_n$, for which $\delta(q_n)=1/q_n$. The bounds for the other values of $q$ are obtained by linear interpolation. 
Despite these similarities, there are some important differences in our proof compared to \cite{FrankSabin-2017clust}, as we will now explain in more detail. Their approach starts by replacing $\Pi_{\lambda}$ by a ``smooth projection'' $\chi(\sqrt{-\Delta_{\g}}-\lambda)$ (see \cite[(18)]{FrankSabin-2017clust}). They then use the Hadamard parametrix to represent the smooth projection as an oscillatory integral operator modulo a smoothing operator. Finally, they apply a generalization of the Carleson--Sj\"olin theorem \cite[Theorem 1]{FrankSabin-2017clust} to prove \eqref{dual estimate in Frank-Sabin}. An analogue of this approach in the one-function case can be found, e.g., in \cite[Chapter 5]{MR3645429}.

Our approach is closer to that of \cite{Smith-2006-C11, Smith-2006sharp, KochSmithTataru-2007-holder} (whose results we generalize in Theorem~\ref{thm:main-result_C2}, Theorem~\ref{thm:main-result_Lip} and Theorem~\ref{thm:main-result_Holder}). As in the one-function case considered in these works, the foundation of our result is the $C^{1,1}$ bound (Theorem \ref{thm:main-result_C2}). We will comment on this case momentarily, after discussing the bounds for $s<2$, and we recall that we consider $q=q_n$ throughout.

 Estimates for $s\in [1,2)$ are derived from the $C^{1,1}$ result together with a frequency-dependent scaling argument and an energy flux estimate of Smith~\cite[equation (15)]{Smith-2006sharp}. On the $R =\lambda^{-(2-s)/(2+s)}$ scale, $C^s$ coefficients are well approximated by $C^{1,1}$ coefficients, which leads to no-loss estimates on cubes of size $R$, as in \eqref{small_ball_estimates-Lip}, and to spectral cluster bounds~\eqref{eq:spectral cluster bounds} on the whole manifold with an $R^{-1/q_n}$ loss compared to the smooth case~\eqref{Frank-Sabin orthoclusters}. The loss arises from summing over roughly $R^{-1}$ slabs of the form $S_R=\{|x_1-c|\leq R\}$, which appear in the energy flux estimates and on which one also has no-loss $L^{q_n}$ bounds. The appearance of $x_1$ as a privileged direction (by a partition of unity and a linear change of coordinates, there is no loss of generality in fixing this direction) comes from an angular localization argument that reduces the Helmholtz operator $-\Delta_{\g}-\lambda^2$ to a frequency-localized first-order pseudodifferential operator (see Section \ref{sec:proof-C2-case}).

For $s<1$, the counterexample of Koch, Smith, and Tataru \cite[Section 4]{KochSmithTataru-2007-holder} shows that eigenfunctions can be exponentially localized to tubes of diameter
$\lambda^{-2/(2+s)}$ and length $\lambda^{s-1}$, so there can be no energy flux bounds above the $T=\lambda^{s-1}$ scale. However, at or below this scale, $C^s$ coefficients are well approximated by Lipschitz ones, which leads to bounds over slabs $S_R$ with a $T^{-1/2}$ loss (see \cite[equation (15)]{KochSmithTataru-2007-holder}). Summing over $TR^{-1}$ disjoint slabs $S_R$ intersecting a given cube of sidelength $T$ and then summing over a cover of $M$ by such cubes yields the $(TR^{-1})^{1/q_n}T^{-1/2}$ loss in~\eqref{eq:spectral cluster bounds}.

We now comment on the $C^{1,1}$ bound. Smith \cite{Smith-2006-C11} obtained the one-function bound as a consequence of square-function (i.e.\ $L^q_xL^2_t$) bounds for the wave equation $(\partial_t^2-\Delta_{\g})u=0$, in the spirit of Mockenhaupt--Seeger--Sogge \cite{MR1173929}. We have not been able to use this approach for orthonormal systems due to an inefficiency of Littlewood--Paley theory in this case. 
Essentially, if $\gamma_k:=P_k\gamma P_k$, where $P_k$ denotes a smooth Littlewood--Paley projection to frequencies of size $2^k$, one is faced, among other things, with estimating expressions like
\begin{align}\label{expression square function bounds}
\sum_k2^{2\delta(q)k}\|\gamma_k\|_{\mathfrak{S}^{\alpha(q)}(L^2(M))},
\end{align}
where $\gamma$ is the operator associated to an orthonormal system of initial data in $H^{\delta(q)}(M)$. For a single initial datum, $\gamma=|f\rangle\langle f|$, this expression reduces to $\|f\|^2_{H^{\delta(q)}(M)}$, and this is one of the reasons why the proof of the main result of \cite{Smith-2006-C11} can be reduced to a frequency-localized estimate \cite[equation (2.12)]{Smith-2006-C11}. However, for a nontrivial orthonormal system consisting of several functions, \eqref{expression square function bounds} cannot be bounded by $\|\gamma\|_{\mathfrak{S}^{\alpha(q)}(H^{\delta(q)}(M))}$ unless $\alpha(q)\leq 1$, which for $\alpha(q)$ as in \eqref{eq-def:alpha} is only true for the trivial case $q=2$. This means that, on the level of operators, a frequency-localized bound analogous to \cite[equation (2.12)]{Smith-2006-C11} cannot readily be summed to yield a global bound. Studying square function bounds for orthonormal systems would be an interesting topic in itself.

The alternative strategy we employ here, while similar in many aspects to \cite{Smith-2006-C11}, does not rely directly on the wave equation. Although square-function bounds are extremely powerful and yield spectral cluster estimates as an immediate corollary (see \cite{Smith-2006-C11}), the latter can also be derived using only the stationary equation. Note, however, that even in the stationary case, locally in phase space, one can always single out a distinguished space direction that plays the role of (fictional) time.
In the single-function setting, this approach is standard and has been used by many authors. For instance, Koch and Tataru \cite{KochTataru-2005} applied it to establish microlocal $L^q$ estimates for quasimodes of pseudodifferential operators, including non-self-adjoint ones. Semiclassical analogues in the self-adjoint case were later obtained by Koch, Tataru, and Zworski \cite{Koch-Tataru-Zworski-2007}. 
In the single-function setting, if the metric is smooth, one can use Fourier integral operators to construct a parametrix (see, e.g., \cite{MR1173929}). Smith \cite{MR1644105,Smith-2006-C11} and Tataru \cite{MR1749052,MR1833146,MR1887639} used wave packet techniques to construct parametrices for regularized wave operators.
The smoothing errors can be treated as forcing terms, and the equation can be solved iteratively. The wave packet parametrices require roughly two derivatives on the coefficients as this ensures uniqueness of the bicharacteristic flow. This approach works equally well for both time-dependent and stationary problems but has not been used before in the context of orthonormal systems.

The key estimate in our proof of the $C^{1,1}$ result is the pointwise kernel bound \eqref{Dunford--Pettis} for the propagator of a regularized, frequency-localized, first-order pseudodifferential operator. This bound follows from a special case of the results of Koch and Tataru \cite[Proposition 4.7]{KochTataru-2005} and the Dunford--Pettis theorem \cite[Theorem 2.2.5]{Dunford-Pettis-1940-Linear}.

\subsection{Remarks on optimality}

There are two aspects of optimality to be considered. First, one can ask whether the exponent $\delta_s(q)$ in the one-function estimate is optimal (smallest possible). Second, if this exponent is optimal, one can ask whether the Schatten exponent~$\alpha(q)$ is optimal (largest possible). We already discussed the first aspect, and we focus now on the second aspect. Since $\dim E_{\lambda}\approx \lambda^{n-1}$ by the sharp Weyl law, the Schatten exponent may always be increased at the expense of increasing $\delta_s(q)$. This is the reason we fix (the optimal) exponent $\delta_s(q)$ as in the one-function bounds.

As explained earlier, for $C^{1, 1}$ coefficients, the exponent~$\delta(q)$ is the same as in the smooth case, hence best possible by the examples in \cite[Section~5.1]{MR3645429}. 
For the Laplacian on the two-dimensional sphere with the standard round metric, optimality of the Schatten exponent $\alpha(q)$ was shown in \cite[Theorem 4]{FrankSabin-2017clust}. They showed optimality in a strong sense, namely, for any fixed ratio $\#J/\lambda^{n-1}$ of the number of functions to the maximal number of orthonormal functions in the cluster $E_{\lambda}$, inequality \eqref{Frank-Sabin orthoclusters} is saturated. There is strong evidence that $\alpha(q)$ is sharp in any dimension $n\geq 2$ and on any compact manifold with smooth metric \cite{FrankSabinpersocomm}. We thus expect that Theorem \ref{thm:main-result_C2} is also optimal in this strong sense.

When $0\leq s<2$, the picture is much less clear. As already mentioned, the one-function bounds in Theorems \ref{thm:main-result_Lip}--\ref{thm:main-result_Holder} (and hence the exponent $\delta_s(q)$) are sharp for $2\leq q\leq q_n$. Therefore, it is relevant to ask whether the Schatten exponent is optimal for this range of $q$. For the sake of this discussion, let us denote the (unknown) optimal Schatten exponent by $\alpha_s(q)$ (our bounds hold with $\alpha_s(q)=\alpha(q)$ independent of $s$). When the number of functions is maximal ($\#J=\dim E_{\lambda}$), the sharp pointwise Weyl law\footnote{To be precise, the sharp pointwise Weyl law holds for $s>1$ by results of Bronstein and Ivrii \cite{MR1974450} (they only require that the first derivatives of the coefficients are Dini--Lipschitz continuous). To the best of our knowledge, the Lipschitz case ($s=1$) is still open.} stipulates that the density is asymptotically equal to $c\lambda^{n-1}$, and hence all of its $L^{q/2}$ norms are comparable. This leads to the necessary condition 
\begin{align}\label{nec. cond. Weyl}
   2\delta_{s}(q) +\frac{n-1}{\alpha_s(q)} \geq  n-1.
\end{align}
Equality in \eqref{nec. cond. Weyl} holds when $\delta_s(q)$ and $\alpha_s(q)$ are replaced by $\delta(q)$ and $\alpha(q)$, respectively. However, since there are losses in the one-function bounds, we have $\delta_s(q)>\delta(q)$, and thus equality in~\eqref{nec. cond. Weyl} would require $\alpha_s(q)>\alpha(q)$. Presently, we do not know if and how the Schatten exponents in Theorems \ref{thm:main-result_Lip} and \ref{thm:main-result_Holder} can be improved. Restricting our attention to the case $s=1$ for the sake of discussion, we point out that \eqref{small_ball_estimates-Lip} and Lemma \ref{lemma rescaling} are still optimal whenever the $C^{1,1}$ bounds are, see Remark \ref{lemma Lip cube sharp}. However, the summation argument \eqref{mess} in Lemma \ref{lemma:energy_estimates_slabs-to-cube}
is probably lossy, see Remark \ref{remark loss in the summation argument}. 

On the flat torus or an open subset of Euclidean space, the optimality of the Schatten exponent can be easily tested by wave packet examples. Considering $n=2$ for simplicity, one such example is when each $u_j$ is a single wavepacket, essentially supported on a $1\times\lambda^{-1/2}$ tube $T$ intersecting the origin. The tubes are $\lambda^{-1/2}$-separated in angle, and there are $\#T\approx\lambda^{1/2}$ many of them. The left-hand side of \eqref{Frank-Sabin orthoclusters} (with all $\nu_j=1$) is then bounded from below by $$ |T|^{-1}\#T|B(0,\lambda^{-1/2})|^{2/q}\approx \lambda^{1-2/q}=\lambda^{2\delta(q)}(\#T)^{1/\alpha(q)},\quad\forall q\geq 6.$$
For $q=6$ (the critical exponent $q_n$ in two dimensions), the lower bound also holds for any intermediate number of functions $1\leq \#T\leq \lambda^{1/2}$. If we repeat this argument with the $1\times \lambda^{-2/3}$ wavepackets constructed in \cite{MR1306017}, then we get a lower bound that matches the upper bound in Theorem \ref{thm:main-result_Lip} for $q\geq 6$. However, the wavepackets in \cite{MR1306017} all point in the $x_1$ direction since they correspond to a special Lipschitz singularity of the type $|x_2|$. We do not know whether more complicated Lipschitz singularities could produce wave packets in all directions and thus justify the above argument.

For special Lipschitz metrics (or manifolds with boundaries) and for $n=2$, the first and third author  \cite{CSuprep} have established a generalization of \cite{MR2316270} to systems of orthonormal functions. These results improve upon the bounds in Theorem \ref{thm:main-result_Lip} in that the exponent of $\lambda$ is smaller than $\delta_1(q)$ for $q>6$ (and equal to $\delta(q)$ for $q>8$), and the Schatten exponent is larger than $\alpha(q)$ for $q\leq 8$. For instance, for $q=6$, the Schatten exponent is $9/5$, whereas $\alpha(6)=3/2$. It is interesting that in the one-function case, the counterexamples for special versus general Lipschitz coefficients are the same, whereas for systems of orthonormal functions this does not seem to be the case. It is an interesting open question whether our Schatten exponents for $s<2$ can be improved, or whether stronger counterexamples can be constructed.

\subsection{Organization of the paper}
We start recalling in  Section \ref{sec:prelim} some standard facts about quadratic forms and operators, Schatten spaces, densities and a duality principle. We also prove that the $L^{\infty}$ bounds are a consequence of the corresponding one-function bounds (\ref{Prop.Linfty}).
In Section \ref{sec:connect-spectral_cluster-resolvent}, we explain the link between clusters estimates and other types of bounds that are more convenient to deal with in our proof. Then, we prove respectively Theorems \ref{thm:main-result_C2}, \ref{thm:main-result_Lip} and \ref{thm:main-result_Holder} in Sections \ref{sec:proof-C2-case}, \ref{sec:proof-Lip-case} and \ref{sec:proof-Holder-case}.

\subsection*{Notation}
For $s\in (0,1)$, $C^s(M)$ denotes the class of H\"older continuous functions on $M$ (bounded for $s=0$ and Lipschitz for $s=1$). For $s\in(1,2]$, we set $C^s(M)=C^{1,s-1}(M)$, the class of functions whose first derivative is in $C^{s-1}(M)$. These spaces are endowed with the norms
\begin{align*}
    \norm{f}_{C^s(M)}:= \norm{f}_{L^{\infty}(M)}+
    \begin{cases}
    \sup_{\abs{h}\leq 1}\abs{h}^{-s}\norm{f(\cdot+h)-f}_{L^{\infty}(M)} &\text{if}\ s\in(0,1],\\
\norm{df}_{L^{\infty}(M)}+\sup_{\abs{h}\leq 1}\abs{h}^{1-s}\norm{df(\cdot+h)-df}_{L^{\infty}(M)} &\text{if}\ s\in(1,2],
    \end{cases}
\end{align*}
and $\norm{f}_{C^0(M)}:= \norm{f}_{L^{\infty}(M)}$.    
The one form $df$ is identified with the vector field $\nabla_{\g}f$ via the Riemannian metric in the usual way, i.e., $df(X)=\g(\nabla_{\g}f, X)$ for any vector field $X$. We say that a metric $\g$ on $M$ is of class $C^s$ if the coefficients $\g_{ij}$ in any local coordinate system are $C^s$. The divergence with respect to the metric is $-\nabla\mathrlap{^*}_{\textup\g}$, and $\nabla\mathrlap{^*}_{\textup\g}\cdot\nabla_{\g}=-\Delta_{\g}$. In local coordinates, $(\nabla_{\g}u)^i=\g^{ij}\partial_ju$ and $\nabla\mathrlap{^*}_{\textup\g}X=|\g|^{-1/2}\partial_i(|\g|^{1/2}X^i)$.
The inner product on $L^2(M)$ is defined as $\langle u,v\rangle:=\int_M \overline{u}v\,|\g|^{1/2}\rd x$.
We adopt the summation convention: repeated indices (one index a subscript, the other one a superscript) are understood to be summed over. 
The notation $A\lesssim B$ indicates $|A|\leq CB$, where the implicit constant $C$ depends only on fixed quantities.
We call $\phi$ a \emph{bump function} adapted to a given set if it is a smooth, compactly supported function, equal to one on that set, and supported on a dilate of the same set.

\section{Preliminaries}\label{sec:prelim}

In this section, we first prove the $L^{\infty}$ bounds. We then recall some well-known facts regarding quadratic forms and operator domains, basic notions of density operators and Schatten spaces, and a duality principle due to Frank and Sabin. We also include some details on the pseudodifferential calculus we use here.

\subsection{The $L^{\infty}$ bound and the case $s=0$}\label{sec:proof-Linfty-estimates}

Smith \cite{Smith-2006sharp} (for $s\in [1,2]$) and Koch--Smith--Tataru \cite{KochSmithTataru-2007-holder} (for $s\in[0,1)$) proved the sharp bounds
\begin{align}\label{Linfty bounds low regularity}
    \|\Pi_{\lambda}\|_{L^2(M)\to L^{\infty}(M)}\lesssim 
    \begin{cases}
        \lambda^{\frac{n-1}{2}} &\text{for }s\in [1,2],\\
        \lambda^{\frac{n-s}{2}} &\text{for } s\in[0,1).
    \end{cases}
\end{align}
These bounds directly imply the case $q=\infty$ in Theorems \ref{thm:main-result_C2}--\ref{thm:main-result_Holder}, as the following abstract result shows.

\begin{proposition}\label{Prop.Linfty}
 Let $(X,\rd x),(Y,\rd y)$ be measure spaces, and let $T:L^2(Y)\to L^\infty(X)$ be a bounded operator with continuous integral kernel. 
If $J$ is a countable index set, $(u_j)_{j\in J}$ an orthonormal system in $L^2(Y)$ and $(\nu_j)_{j\in J}\subset \C$, then
\begin{equation*}
    \|\sum_{j\in J} \nu_j|Tu_j|^2\|_{L^{\infty}(X)}\leq \|T\|_{L^2(Y)\to L^{\infty}(X)}^2\sup_{j\in J} |\nu_j|.
\end{equation*}
\end{proposition}

\begin{proof}
Without loss of generality, we may assume that $\sup_{j\in J} |\nu_j|=1$. Let $K:X\times Y\to\C$ denote the integral kernel of $T$. For fixed $x\in X$, set $K^x(y)=K(x,y)$. Then
    $$
Tu_j(x)=\int_Y K^x(y)u_j(y)\rd y=\langle K^x,u_j\rangle_{L^2(Y)}.
    $$
By Bessel's inequality,
$$
\sum_{j=1}^N \nu_j|Tu_j(x)|^2=\sum_{j=1}^N \nu_j|\langle K^x,u_j\rangle_{L^2(Y)}|^2\leq  \|K^x\|_{L^2(Y)}^2.
$$
Taking the supremum over $x$ yields
\begin{align}
    \sup_{x\in X}\sum_{j=1}^N |Tu_j(x)|^2\leq \sup_{x\in X}\int_Y |K(x,y)|^2\rd y=\|T\|_{L^2(Y)\to L^{\infty}(X)}^2,
\end{align}
completing the proof.
\end{proof}

If $s=0$, then $\delta_0(q)=n(1/2-1/q)$ for $2\leq q\leq \infty$. Hence, in that case, Theorem \ref{thm:main-result_Holder} already follows by interpolating the $L^2$ and the $L^{\infty}$ bound.

\subsection{Quadratic forms and operators}\label{subsec. quad. form}
If $\g$ is smooth, then $-\Delta_g$ is essentially self-adjoint on $C^{\infty}(M)$ (i.e., has a unique self-adjoint extension). In the case of limited regularity, we will define a self-adjoint extension by quadratic form methods. Here, we only assume that the coefficients of the metric are bounded and measurable. Consider the sesquilinear form
\begin{align}
Q(u,v):=\int_M\g^{ij}\partial_i\overline{u}\partial_j v\,|\g|^{1/2}\rd x,\quad u,v\in C^{\infty}(M),
\end{align}
and denote by $Q(u)=Q(u,u)$ the corresponding quadratic form.
Since $(\g_{ij})$ is positive definite, there exists $c_{\rm ell}>0$ such that, in local coordinates near $x\in M$, one has for all $\xi\in \R^n$
\begin{align}\label{ellipticity g}
c_{\rm ell}^{-1}|\xi|^2\leq \g^{ij}(x)\xi_i\xi_j\leq c_{\rm ell} |\xi|^2, 
\end{align}
It follows from standard results (e.g., \cite[Theorem 1.2.6]{MR1103113}) that $Q$ is closable. We will denote the closure again by $Q$. Then there exists a canonical non-negative self-adjoint operator $P$ associated to $Q$, with form domain 
\begin{align}
   \dom(Q)=\dom(\sqrt{P})=\{u\in L^2(M):\int_M\g^{ij}\partial_i\overline{u}\partial_j u\,|\g|^{1/2}\rd x<\infty\}=H^1(M),
\end{align}
and $Q(u)=\|\sqrt{P}u\|_{L^2(M)}^2$,
see, e.g., \cite[Theorem 1.2.1]{MR1103113}, \cite[Theorems VI.2.1, VI.2.6, VI.2.23]{MR407617}.
Here, the derivatives are taken in the weak sense.
The operator domain of $P$ is given by
\begin{align}\label{dom(P)}
    \dom(P)=\{u\in H^1(M):\exists f\in L^2(M) \mbox{ s.t. } 
   Q(u,v)=\langle f,v\rangle, 
   \forall v\in H^1(M)\},
\end{align}
and $Pu=f$, see, e.g., \cite[Theorem 1.2.7]{MR1103113}, \cite[Theorem 2.14]{MR3243083}. If the coefficients $\g^{ij}$ are $C^1$, then $\dom(P)=H^2(M)$, and 
\begin{align}\label{P for Lipschitz}
    Pu=-|\g|^{-1/2}\partial_i(|\g|^{1/2}\g^{ij}\partial_j u)=-\g^{ij}\partial_i\partial_j u-|\g|^{-1/2}\partial_i(|\g|^{1/2}\g^{ij})\partial_j u,\quad \forall u\in H^2(M).
\end{align}
We will call $P$ the Laplace–Beltrami operator (or simply Laplacian) on $M$. By the Rellich–Kondrachov theorem, $P$ has compact resolvent. We denote the sequence of discrete eigenvalues (counted with multiplicity) by $(\lambda_j^2)_{j=1}^{\infty}$ and the sequence of corresponding eigenfunctions by $(e_j)_{j=1}^{\infty}$.
Then  
\begin{align}\label{Pilambda and Elambda of P}
\Pi_{\lambda}:=\mathds{1}(\sqrt{P}\in [\lambda,\lambda+\lambda^{(1-s)_+}]):=\sum_{\lambda_j\in [\lambda,\lambda+\lambda^{(1-s)_+}]}|e_j\rangle\langle e_j|.
\end{align}
The complex square root is defined by the spectral theorem, i.e., 
\begin{align}
(P-z)^{-1/2}=\sum_{j=1}^{\infty}(\lambda_j^2-z)^{-1/2}|e_j\rangle\langle e_j|,
\end{align}
with $\zeta^{1/2}:=|\zeta|^{1/2}\e^{\frac{\I}{2}\arg \zeta}$, $0\leq \arg \zeta<2\pi$, for all $\zeta\in\C\setminus\{0\}$.

\subsection{Schatten spaces}\label{subsec:Schatten}

If $\Hc$ and $\Hc'$ are complex separable Hilbert spaces and $\alpha\in [1,\infty]$, the Schatten space $\schatten^\alpha(\Hc,\Hc')$ is defined as the space of compact linear operators $A:\Hc\to\Hc'$ with norm 
	\begin{equation*}
		\normSch{A}{\alpha}{(\Hc,\Hc')} := \left(\tr_\Hc(A^*A)^{\alpha/2}\right)^{1/\alpha}=\left(\tr_{\Hc'}(AA^*)^{\alpha/2}\right)^{1/\alpha}
	\end{equation*}
if $\alpha<\infty$ and with the operator norm $\|\cdot\|$ if $\alpha=\infty$. 
When $\Hc=\Hc'$, we write $\schatten^{\alpha}(\Hc)$ instead of $\schatten^{\alpha}(\Hc,\Hc')$. We sometimes omit $\Hc,\Hc'$ if the spaces are clear from the context. For background on Schatten spaces, we refer for instance, to \cite{Simon-2005}.

In the following, we consider operators on $\Hc=L^2(M)$ or $L^2(\R^n)$, and the corresponding Schatten spaces.
There are obvious analogues for $L^2(\Omega)$, where $\Omega$ is an open subset of $M$ or $\R^n$, or more generally, an arbitrary $\sigma$-finite measure space.
 
We will repeatedly use 
the triangle and H\"older's inequality in Schatten spaces. The latter states that
\begin{align}\label{Holder ineq. in Schatten}
    \|AB\|_{\mathfrak{S}^{\alpha_1}(\Hc)}\leq \|A\|_{\mathfrak{S}^{\alpha_2}(\Hc)} 
    \|B\|_{\mathfrak{S}^{\alpha_3}(\Hc)}\quad \mbox{for}\quad \alpha_1^{-1}=\alpha_2^{-1}+\alpha_3^{-1}.
\end{align}
We also recall the Kato--Seiler--Simon inequality \cite[Theorem 4.1]{Simon-2005}: 
For $2\leq p\leq \infty$,
    \begin{equation}\label{KSS inequality}
        \normSch{W(x)\beta(-{\mathrm i}\nabla)}{p}{\left(L^2(\R^n)\right)}\leq (2\pi)^{-n/p} \normLp{W}{p}{(\R^n)}\normLp{\alpha}{p}{(\R^n)}.
    \end{equation}

Frank and Sabin \cite[Appendix C]{FrankSabin-2017clust} proved a version of the Kato--Seiler--Simon inequality for the Laplacian on compact manifolds. Here we state a version for low-regularity metrics. 
\begin{proposition}\label{Prop. KSS low reg.}
 Let $s\in[0,2]$, and let $\g$ be of class $C^s(M)$. Then there exists $C>0$ such that for all $2\leq p\leq \infty$,
 \begin{align}
\|W\beta(\sqrt{P})\|_{\mathfrak{S}^{p}(L^2(M))}\leq C\|W\|_{L^{p}(M)}\bigg(\sum_{k=0}^{\infty}\sup_{k\leq \tau\leq k+1}|\beta(\tau)|^p(1+k)^{n-1+(1-s)_+}\bigg)^{1/p}.
 \end{align}
\end{proposition}
\begin{proof}
The case $p=\infty$ is trivial, so we focus on $p<\infty$. The proof of \cite[Proposition 24]{FrankSabin-2017clust} carries over with minor modifications. By the Lieb-Thirring inequality (see \cite{FrankSabin-2017clust} for references),
\begin{align}
\|W\beta(\sqrt{P})\|_{\mathfrak{S}^p(L^2(M))}^p&\leq\tr(|W|^p|\beta(\sqrt{P})|^p)=\sum_{k=0}^{\infty}\tr(|W|^p|\beta(\sqrt{P})|^p\mathbf{1}(\sqrt{P}\in[k,k+1]))\\
&\leq \sum_{k=0}^{\infty}\sup_{k\leq\tau\leq k+1}|\beta(\tau)|^p\tr (\mathbf{1}(\sqrt{P}\in[k,k+1])|W|^p)\\
&\leq \|W\|_{L^{p}(M)}^p
\sum_{k=0}^{\infty}\sup_{k\leq\tau\leq k+1}|\beta(\tau)|^p \|\mathbf{1}(\sqrt{P}\in[k,k+1])\|_{L^2(M)\to L^{\infty}(M)}^2.
\end{align}
The claim now follows from the estimates \eqref{Linfty bounds low regularity}.
\end{proof}

\begin{remark}
 For $s<1$, we could also have split the sum over $k$ into clusters of size $k^{1-s}$ since the same bounds hold for the $L^2\to L^{\infty}$ norm of the corresponding projection operator as for clusters of unit size. For instance, if $s=0$, we would then obtain
 \begin{align}
\|W\beta(\sqrt{P})\|_{\mathfrak{S}^{p}(L^2(M))}\leq C\|W\|_{L^{p}(M)}\bigg(\sum_{k=0}^{\infty}\sup_{2^k-1\leq \tau\leq 2^{k+1}}|\beta(\tau)|^p2^{kn}\bigg)^{1/p}.
 \end{align}
However, the result of Proposition \ref{Prop. KSS low reg.} will be sufficient for our purposes.
\end{remark}

When applying Proposition \ref{Prop. KSS low reg.}, we will often need to estimate sums of the following type.

\begin{lemma}\label{lemma sums a,b}
 Let $a,b\geq 0$. Then for $\lambda\geq 2$,   \begin{align}
\sum_{k=0}^{2\lambda}(1+|k-\lambda|)^{-a}(1+k)^b\lesssim
     \begin{cases}
\lambda^{b+1-a}\quad&\mbox{if }a<1,\\
\lambda^{b}\log\lambda\quad&\mbox{if }a=1,\\
\lambda^{b}\quad&\mbox{if }a>1,
     \end{cases}    
 \end{align}
and if $ b - 2a < -1 $, 
\begin{align}
\sum_{k>2\lambda}(k^2+\lambda^2)^{-a}(1+k)^b\lesssim \lambda^{b+1-2a}.    
\end{align} 
\end{lemma}

\subsection{Densities}

If $\gamma$ is a non-negative finite rank operator on $L^2(M)$, then it has a spectral decomposition 
\begin{align}\label{density operator gamma}
    \gamma=\sum_{j\in J}\nu_j|u_j\rangle\langle u_j|,
\end{align}
where $\# J<\infty$, $\nu_j\geq 0$ and $u_j$ are normalized in $L^2(M)$. If, in addition, $\gamma=\Pi_{\lambda}\gamma \Pi_{\lambda}$, then $u_j\in E_{\lambda}$, i.e. each $u_j$ is a finite linear combination of eigenfunctions of $P$. If the coefficients of the metric are bounded and measurable, elliptic regularity (see, e.g., \cite[Theorems 8.15, 8.24]{MR737190} and the comment after \cite[Theorem 8.37]{MR737190}) implies that $u_j\in C^{\alpha}(M)$ for some $\alpha>0$. Then, the \textit{density} is pointwise defined by
\begin{align*}
    \rho_{\gamma}(x):=\sum_{j\in J}\nu_j|u_j(x)|^2,\quad x\in M,
\end{align*}
and also belongs to
$C^{\alpha}(M)$ for the same $\alpha$. In particular, $\rho_{\gamma}\in L^{q/2}(M)$ for any $q\in [2,\infty]$. 

We will also need to consider the density of the conjugated operator $A\gamma A^*$, where $A$ is some bounded operator on $L^2(M)$. In this case, 
\begin{align}
  \rho_{A\gamma A^*}:=\sum_{j\in J}\nu_j |Au_j|^2 
\end{align}
is a priori only in $L^1(M)$, but we still have
\begin{align*}
    \int_M \rho_{A\gamma A^*} |\g|^{1/2}\rd x = \tr(A\gamma A^*). 
\end{align*}
Indeed, if $(f_k)_k$ is any orthonormal basis of $L^2(M)$, then
    \begin{align*}
    \tr(A\gamma A^*)&=\sum_{k=1}^{\infty}\langle A\gamma A^* f_k,f_k\rangle
        =\sum_{k=1}^{\infty}\sum_{j\in J}\nu_j |\langle Au_j,f_k\rangle|^2\\
        &=\sum_{j\in J}\nu_j\sum_{k=1}^{\infty} |\langle Au_j,f_k\rangle|^2
        =\sum_{j\in J}\nu_j\|Au_j\|_{L^2(M)}^2=\int_M \rho_{A\gamma A^*}|\g|^{1/2}\rd x,
    \end{align*}
where the second to last equality follows by Bessel's identity. Similarly, if $W$ is a  measurable function, 
\begin{align}\label{density of AgammaA*}
  \int_M \rho_{A\gamma A^*}|W|^2|\g|^{1/2}\rd x = \tr(WA\gamma A^*\overline{W}).   
\end{align}

\subsection{Duality}\label{subsection duality}

We first recall the duality principle \cite[Lemma 3]{FrankSabin-2017rest} of Frank and Sabin.

\begin{lemma}\label{lemma duality principle}
  Let $A:L^2(M) \to L^q(M)$ be a bounded operator for some $2 \leq q <\infty$, and let $\alpha\geq 1$. Then the following bounds are equivalent:
\begin{align}\label{Frank-Sabin duality 1}
 \normLp{\rho_{A\gamma A^*}}{q/2}{(M)}\leq C\normSch{\gamma}{\alpha}{(L^2(M))}
\end{align}
for any non-negative finite rank operator $\gamma$ on $L^2(M)$, and
\begin{align}\label{Frank-Sabin duality 2}
\normSch{WAA^*\overline{W}}{\alpha'}{(L^2(M))}\leq C' \|W\|_{L^{2(q/2)'}(M)}^2, \quad \forall W \in L^{2(q/2)'}(M).    
\end{align}
Moreover, the values of the optimal constants $C$ and $C'$
coincide.
\end{lemma}

The lemma follows from H\"older's inequality and the identity
\begin{align}\label{eq:duality}
    \normLp{\rho_{A\gamma A^*}}{q/2}{(M)}=\sup_{W\in L^{2(q/2)'}(M)}\frac{\tr(WA\gamma A^*\overline{W})}{\normLp{W}{2(q/2)'}{(M)}^2}=\sup_{\normLp{W}{2(q/2)'}{(M)}=1}\normSch{WA\sqrt{\gamma}}{2}{(L^2(M))}^2.
\end{align}
The latter also yields a useful version of the triangle inequality.

\begin{lemma}\label{lemma triangle inequality}
Let $2\leq q\leq \infty$.
Assume that $(A_j)_{j=1}^K$ is a family of bounded operators $L^2(M)\to L^q(M)$, and let $A=\sum_{j=1}^K A_j$. If $\gamma$ is a non-negative finite-rank operator on $L^2(M)$, then
\begin{align}\label{triangle ineq. for gamma}
\normLp{\rho_{A\gamma A^*}}{q/2}{(M)}\leq K\sum_{j=1}^K \normLp{\rho_{A_j\gamma A_j^*}}{q/2}{(M)}.
\end{align}
\end{lemma}

\begin{proof}
    By applying \eqref{eq:duality} twice,
    \begin{align*}
        \normLp{\rho_{A\gamma A^*}}{q/2}{(M)}&=\sup_{\normLp{W}{2(q/2)'}{(M)}=1}\normSch{WA\sqrt{\gamma}}{2}{(L^2(M))}^2 \\
        &\leq \sup_{\normLp{W}{2(q/2)'}{(M)}=1}\Big(\sum_{j=1}^K\normSch{WA_j \sqrt{\gamma}}{2}{(L^2(M))}\Big)^2\\
        &
        \leq \Big(\sum_{j=1}^K \normLp{\rho_{A_j\gamma A_j^*}}{q/2}{(M)}^{1/2}\Big)^2\\
        &\leq K\sum_{j=1}^K \normLp{\rho_{A_j\gamma A_j^*}}{q/2}{(M)},
    \end{align*}
where we used the Cauchy--Schwarz inequality in the last line.    
\end{proof}

\subsection{Pseudodifferential calculus}
We will use a modest amount of pseudodifferential calculus.
We will work with the following fixed-frequency symbol classes, as in \cite{KochTataru-2005,MR3282983}. 
\begin{definition}[Symbol classes]
Let $a(t,x,\xi,\lambda)$ be a family of symbols in $(x,\xi)\in T^*\R^d$, depending on parameters $t\in\R$ and $\lambda\geq 1$. For $\sigma\in [0,1]$, we say $a(t,x,\xi,\lambda)\in S_{\lambda,\lambda^{\sigma}}$ if, for all multi-indices $\alpha,\beta$, there exist constants $C_{\alpha,\beta}$ such that
\begin{align}\label{symbols Slambda 1}
|\partial_{t,x}^{\alpha}\partial_{\xi}^{\alpha}a(t,x,\xi,\lambda)|&\leq C_{\alpha,\beta}\lambda^{-|\beta|+\sigma |\alpha|}.
\end{align}
We say $a(t,x,\xi,\lambda)\in C^kS_{\lambda,\lambda^{\sigma}}$ if the stronger bound 
\begin{align}\label{symbols Slambda 2}
|\partial_{t,x}^{\alpha}\partial_{\xi}^{\alpha}a(t,x,\xi,\lambda)|&\leq C_{\alpha,\beta}\lambda^{-|\beta|+\sigma(|\alpha|-k)_+}
\end{align}
holds, where $x_+=\max(x,0)$. 
\end{definition}
These are symbols of order zero. For symbols of order $m\in\R$ we write $a\in \lambda^mS_{\lambda,\lambda^{\sigma}}$.
We will only consider the cases $\sigma=1$, $k=1$ and $\sigma=1/2$, $k=2$. The dimension $d$ will either be $n$ or $n-1$. 
In the first case, there will be no parameter $t$. In the second case, $t=x_1$ will be one of the spatial variables.
The symbol class denoted here by $C^2S_{\lambda,\lambda^{1/2}}$ corresponds to $S^2_{\lambda}$ in \cite{KochTataru-2005}. We have adopted the terminology of \cite[Definition 2.2]{MR3282983} here. Symbols will initially be defined in the phase space region
\begin{align}
    B_{\lambda}:=\{(x,\xi)\in T^*\R^n:|x|\leq 1,|\xi|\leq \lambda\}.
\end{align}
and extended globally to $\R^d$, consistently with the symbol class.

As noted in \cite{MR3282983}, $a\in S_{\lambda,\lambda}$ constitutes a bounded family of zero-order symbols in the Hörmander class $S^0_{0,0}$ when rescaled via $(t,x,\xi)\mapsto (\lambda t,\lambda x,\xi/\lambda)$. Consequently, both the \emph{Kohn-Nirenberg quantization} $a(t,x,D_x)$ and the \emph{Weyl quantization} $a^w(t,x,D_x)$ are uniformly bounded on $L^2$, with bounds independent of $t$ and $\lambda$ \cite{MR517939}. In what follows, we adopt the Weyl quantization
\begin{align*}
    a^w(t,x,D,\lambda)u(x)=(2\pi)^{-n}\int_{\R^n}\int_{\R^n}\e^{\I (x-y)\cdot\xi}a\Big(t,\frac{x+y}{2},\xi,\lambda\Big)u(y)\rd y\rd\xi,\quad \forall u\in \mathcal{S}(\R^d).
\end{align*}
This is convenient but not essential. 
We usually simply write $a^w$ instead of $ a^w(t,x,D,\lambda)$. The symbol classes $S_{\lambda,\lambda^{\sigma}}$ as well as the corresponding classes of operators $OPS_{\lambda,\lambda^{\sigma}}$ are algebras. If $\sigma<1$, the usual pseudodifferential calculus applies. For symbols in $a,b\in C^1S_{\lambda,\lambda}$,
the asymptotic formulas for composition and adjoint remain valid to first-order,
\begin{align}\label{first order calculus}
a^wb^w=(ab)^w+\lambda^{-1}R_1,\quad(a^w)^*=(\overline{a})^w+\lambda^{-1}R_2,
\end{align}
where the remainders $R_1,R_2$ are Weyl quantizations of symbols in $S_{\lambda,\lambda}$ and are therefore bounded on $L^2(\mathbb{R}^{d})$.
We call $R$ a smoothing operator if $R\in\lambda^{-N}OPS_{\lambda,\lambda}$ for any $N>0$, in which case we have $\|R\|\lesssim\lambda^{-N}$ for any $N>0$. However, the fixed-frequency calculus does not readily imply trace norm estimates for $R$, and we resort to the standard $S^m_{1,0}$ calculus for this purpose. Precisely, if $r\in S^{-\infty}_{1,0}$, then
\begin{align}\label{smoothing operator trace class}
\|r^w(x,\lambda^{-1}D)\|_{\mathfrak{S}^1(L^2(\R^d))}\lesssim \lambda^{-N}.
\end{align}
This follows from the translation-invariant kernel bound
\begin{align}
    |r^w(x,y)|\leq C_N\lambda^{-N}(1+\lambda|x-y|)^{-N},
\end{align}
where $r^w(x,y)$ denotes the kernel of $r^w$.

\section{Connection between spectral cluster and resolvent estimates}\label{sec:connect-spectral_cluster-resolvent}

In this section, we elucidate the connection between spectral cluster estimates, quasimode estimates, and two versions of resolvent estimates. We consider spectral clusters $[\lambda,\lambda+\mu]$ of size $\mu=\mu(\lambda)$. In the application to our problems, we will choose $\mu(\lambda)=\lambda^{(1-s)_+}$.

\begin{proposition}\label{prop. equivalence of cluster, qm and resolvent}
Let $P$ be a non-negative self-adjoint operator on $L^2(M)$, $1\leq \mu\leq \lambda$, and let $\Pi_{\lambda,\mu}:=\mathds{1}(\sqrt{P}\in [\lambda,\lambda+\mu])$ and $\Pi_{\leq2\lambda}:=\mathds{1}(\sqrt{P}\in [0,2\lambda])$. Let $q\in[2,\infty]$, $\delta\geq 0$ and $\alpha\geq 1$. Then the following are equivalent:
        \begin{align}\label{equivalence of cluster, qm and resolvent 1}
\normLp{\rho(\Pi_{\lambda,\mu}\gamma\Pi_{\lambda,\mu})}{q/2}{(M)}&\lesssim\lambda^{2\delta}\normSch{\gamma}{\alpha}{},\\
        \label{equivalence of cluster, qm and resolvent 2}
    \|\rho(\Pi_{\leq2\lambda}\gamma\Pi_{\leq2\lambda})\|_{L^{q/2}(M)}&\lesssim\lambda^{2\delta}\left(\normSch{\gamma}{\alpha}{}+(\mu\lambda)^{-2}\normSch{(P-\lambda^2)\gamma(P-\lambda^2)}{\alpha}{}\right),\\
\label{equivalence of cluster, qm and resolvent 3}
     \|\rho\big((P-(\lambda+\I\mu)^2)^{-1}&\Pi_{\leq2\lambda}\gamma \Pi_{\leq 2\lambda} (P-(\lambda-\I\mu)^2)^{-1}\big)\|_{L^{q/2}(M)}
     \lesssim\lambda^{2\delta-2}\mu^{-2}\normSch{\gamma}{\alpha}{}.
 \end{align}
    Here, $\gamma$ is a non-negative finite-rank operator, and $\Ran\gamma\subset\dom(P)$ in \eqref{equivalence of cluster, qm and resolvent 2}. The implicit constants are independent of $\gamma,\lambda,\mu$. Moreover, 
the above estimates are implied by the following resolvent bound,
\begin{align}\label{equivalence of cluster, qm and resolvent 4}
 \normSch{W_1(P-(\lambda\pm\I\mu)^2)^{-1}\Pi_{\leq2\lambda}W_2}{\alpha'}{}
 \lesssim\lambda^{2\delta-1}\mu^{-1}\|W_1\|_{L^{2(q/2)'}(M)}  \|W_2\|_{L^{2(q/2)'}(M)}   
\end{align}
for all $W_1,W_2 \in L^{2(q/2)'}(M)$, and for both choices of the sign $\pm$.
\end{proposition}

\begin{remark}
Since the resolvent $(P-(\lambda\pm\I\mu)^2)^{-1}$ is not of the form $AA^*$, the duality principle \cite[Lemma 3]{FrankSabin-2017rest} does not apply to \eqref{equivalence of cluster, qm and resolvent 4}, i.e., there is no equivalent formulation in terms of a density.    
\end{remark}

\begin{proof}[Proof of Proposition \ref{prop. equivalence of cluster, qm and resolvent}]
\item[\eqref{equivalence of cluster, qm and resolvent 2}$\implies$\eqref{equivalence of cluster, qm and resolvent 1}:] Since $\Ran\Pi_{\lambda,\mu}\subset\Ran \Pi_{\leq2\lambda}$, this follows from H\"older's inequality in Schatten spaces,
    \begin{align}
    \normSch{(P-\lambda^2)\Pi_{\lambda,\mu}\gamma \Pi_{\lambda,\mu}(P-\lambda^2)}{\alpha}{}\leq  \|\Pi_{\lambda,\mu}(P-\lambda^2)\|^2   \normSch{\gamma}{\alpha}{}\lesssim (\mu\lambda)^2 \normSch{\gamma}{\alpha}{}.
    \end{align}
\item[\eqref{equivalence of cluster, qm and resolvent 3}$\implies$\eqref{equivalence of cluster, qm and resolvent 2}:] Since $P$ and $\Pi_{\leq2\lambda}$ commute, we have
\begin{align}
    \Pi_{\leq2\lambda}\gamma \Pi_{\leq2\lambda}
    &=(P-(\lambda+\I\mu)^2)^{-1}\Pi_{\leq2\lambda}(P-(\lambda+\I\mu)^2)\gamma (P-(\lambda-\I\mu)^2) \Pi_{\leq 2\lambda} (P-(\lambda-\I\mu)^2)^{-1}
\end{align}
whenever $\Ran\gamma\subset\dom(P)$, and therefore, 
\begin{align}
 \|\rho(\Pi_{\leq2\lambda}&\gamma \Pi_{\leq2\lambda})\|_{L^{q/2}(M)}\\
 &\lesssim  
 \lambda^{2\delta-2}\mu^{-2}\normSch{(P-(\lambda+\I\mu)^2)\gamma (P-(\lambda-\I\mu)^2)}{\alpha}{}\\
 & =  
 \lambda^{2\delta-2}\mu^{-2}\normSch{(P-\lambda^2-2\I\lambda^{1+(1-s)_+}+\lambda^{2(1-s)_+})\sqrt{\gamma}}{2\alpha}{}^2\\
 &\lesssim \lambda^{2\delta-2}\mu^{-2}(\normSch{(P-\lambda^2)\sqrt{\gamma}}{2\alpha}{}^2+\lambda^{2(1+(1-s)_+)}\normSch{\sqrt{\gamma}}{2\alpha}{}^2)\\
 &=\lambda^{2\delta}(\normSch{\gamma}{\alpha}{}+(\mu\lambda)^{-2}\normSch{(P-\lambda^2)\gamma(P-\lambda^2)}{\alpha}{}),
\end{align}
where we used \eqref{equivalence of cluster, qm and resolvent 3} in the first inequality, and the triangle inequality
in the second inequality.\\

\item[\eqref{equivalence of cluster, qm and resolvent 1}$\implies$\eqref{equivalence of cluster, qm and resolvent 3}:] 
By duality (Lemma \ref{lemma duality principle}), with $A=(P-(\lambda+\I\mu)^2)^{-1}\Pi_{\leq2\lambda}$ there, it suffices to prove 
\begin{align}\label{dual of TTstar resolvent}
 &\normSch{W(P-(\lambda+\I\mu)^2)^{-1}\Pi_{\leq2\lambda}(P-(\lambda-\I\mu)^2)^{-1}\overline{W}}{\alpha'}{}
 \lesssim\lambda^{2\delta-2}\mu^{-2}\|W\|^2_{L^{2(q/2)'}(M)}       
\end{align}
for all $W \in L^{2(q/2)'}(M)$,
where we observed that $A^*=\Pi_{\leq2\lambda}(P-(\lambda-\I\mu)^2)^{-1}$ and that $\Pi_{\leq2\lambda}^*=\Pi_{\leq2\lambda}=\Pi_{\leq2\lambda}^2$.
By Schatten space duality (see, e.g., \cite[Theorem 3.2]{Simon-2005}) the latter is equivalent to proving
\begin{align}
    \tr(WAA^*\overline{W}\gamma)\lesssim \lambda^{2\delta-2}\mu^{-2}\|W\|^2_{L^{2(q/2)'}(M)}\|\gamma\|_{\mathfrak{S^{\alpha}}}
\end{align}
for all $W \in L^{2(q/2)'}(M)$ and finite-rank $\gamma\geq 0$. 
Using the spectral decomposition \eqref{density operator gamma} and completing the system $\{u_j\}_{j\in J}$ there to an orthonormal basis of $L^2(M)$, we compute the trace
\begin{align*}
\tr(WAA^*\overline{W}\gamma)
=\sum_{j\in J}\langle u_j,WAA^*\overline{W}\gamma u_j\rangle
=\sum_{j\in J}\nu_j \|A^*\overline{W}u_j\|^2.
\end{align*}
Let $(e_k)_{k=1}^{\infty}, (\lambda_k^2)_{k=1}^{\infty}$ denote the eigenvectors and eigenvalues of $P$, i.e.\,  $Pe_k=\lambda_k^2e_k$. Expanding the last display in eigenvectors, we have 
\begin{align*}
\tr(WAA^*\overline{W}\gamma)&=\sum_{j\in J}\sum_{k=1}^{\infty}\nu_j |\langle e_k,A^*\overline{W}u_j\rangle|^2
\\&=\sum_{j\in J}\sum_{\lambda_k\leq 2\lambda}\nu_j |\lambda_k^2-(\lambda+\I\mu)^2|^{-2} |\langle e_k,\overline{W}u_j\rangle|^2.
\end{align*}
For $\ell\in \mu\Z$, $0\leq \ell\leq 2\lambda$ and $\lambda_k\in [\ell,\ell+\mu]$, we have 
\begin{align}
    |\lambda_k^2-(\lambda+\I\mu)^2|^{-2}\lesssim \lambda^{-2}(\mu+|\ell-\lambda|)^{-2}.
\end{align}
Thus, by Bessel's identity and Tonelli's theorem,
\begin{align*}
\tr(WAA^*\overline{W}\gamma)
&\lesssim \lambda^{-2}\sum_{j\in J}\sum_{\ell\in \mu\Z\cap[0,2\lambda]}\sum_{k\:\lambda_k\in [\ell,\ell+\mu]}\nu_j(\mu+|\ell-\lambda|)^{-2} |\langle e_k,\overline{W}u_j\rangle|^2\\
&=\lambda^{-2}\sum_{j}\sum_{\ell}\nu_j(\mu+|\ell-\lambda|)^{-2}\|\Pi_{[\ell,\ell+\mu]}\overline{W}u_j\|^2_{L^2}\\
&=\lambda^{-2}\sum_{\ell}(\mu+|\ell-\lambda|)^{-2}\sum_{j}\nu_j\|\Pi_{[\ell,\ell+\mu]}\overline{W}u_j\|^2_{L^2}\\
&=\lambda^{-2}\sum_{\ell}(\mu+|\ell-\lambda|)^{-2}\tr(\Pi_{[\ell,\ell+\mu]}|W|^2\Pi_{[\ell,\ell+\mu]}\gamma)\\
&\leq\lambda^{-2}\sum_{\ell}(\mu+|\ell-\lambda|)^{-2}\|\Pi_{[\ell,\ell+\mu]}|W|^2\Pi_{[\ell,\ell+\mu]}\|_{\mathfrak{S}^{\alpha'}}\|\gamma\|_{_{\mathfrak{S}^{\alpha}}}\\
&\lesssim \lambda^{2\delta-2}\mu^{-2}\|\gamma\|_{_{\mathfrak{S}^{\alpha}}}\|W\|^2_{L^{2(q/2)'}},
\end{align*}
where in the last inequality, we used the dual version of \eqref{equivalence of cluster, qm and resolvent 1} to bound
\begin{align*}
\|\Pi_{[\ell,\ell+\mu]}|W|^2\Pi_{[\ell,\ell+\mu]}\|_{\mathfrak{S}^{\alpha'}}=\|W\Pi_{[\ell,\ell+\mu]}\overline{W}\|_{\mathfrak{S}^{\alpha'}}   \lesssim \lambda^{2\delta}\|W\|^2_{L^{2(q/2)'}},\quad \forall \ell\leq 2\lambda.
\end{align*}
\item[\eqref{equivalence of cluster, qm and resolvent 4}$\implies$\eqref{equivalence of cluster, qm and resolvent 3}] follows from the resolvent identity.
\end{proof}

\begin{remark}\label{remark equivalence with A}
An inspection of the proof 
shows that for any bounded operator $A$ on $L^2(M)$, the following are equivalent:
\begin{align}
\|\rho(A\gamma A^*)\|_{L^{q/2}(M)}\lesssim\lambda^{2\delta}\left(\normSch{\gamma}{\alpha}{}+(\mu\lambda)^{-2}\|(P-\lambda^2)\gamma(P-\lambda^2)\|_{\schatten^\alpha}\right),
\end{align}
\begin{align}
\|\rho\left(A(P-(\lambda+\I\mu)^2)^{-1}\gamma (P-(\lambda-\I\mu)^2)^{-1}A^*\right)\|_{L^{q/2}(M)}\lesssim\lambda^{2\delta-2}\mu^{-2}\normSch{\gamma}{\alpha}{}.
\end{align}
\end{remark}

The next lemma shows that, if $P$ is the Laplacian and $\mu=1$, then we may add a gradient term $\lambda^{-2}\normSch{\nabla_{\g}\gamma\nabla\mathrlap{^*}_{\textup\g}}{\alpha}{}$ to the expression in the parenthesis of \eqref{equivalence of cluster, qm and resolvent 2}.
This form is convenient since gradient terms naturally arise in localization arguments (see Subsection \ref{sect. Reduction to Euclidean space}).

\begin{lemma}\label{lemma gradient term}
Let $P$ be the Laplacian, defined in the sense of quadratic forms, with bounded and measurable coefficients (see Subsection \ref{subsec. quad. form}), and 
let $\alpha\geq 1$. Then
\begin{align*}
\lambda^{-2}\normSch{\nabla_{\g}\gamma \nabla\mathrlap{^*}_{\textup\g}}{\alpha}{(L^2(M))}
    \lesssim 
\normSch{\gamma}{\alpha}{(L^2(M))}+\lambda^{-2}\normSch{(P-\lambda^2)\gamma(P-\lambda^2)}{\alpha}{(L^2(M))}.
\end{align*}
for any finite rank operator $\gamma\geq 0$ with $\Ran\gamma\subset\dom(P)$ and for any $\lambda\geq 1$.
\end{lemma}
\begin{proof}
For $f\in\dom(\sqrt{P})$,
\begin{align*}
\|\nabla_{\g}f\|_{L^2(M)}^2=Q(f)=\|\sqrt{P}f\|^2_{L^2(M)}\implies \|\nabla_{\g}(P+\lambda^2)^{-1/2}\|\leq 1.
\end{align*}
Thus,
\begin{align}
\lambda^{-2} \normSch{\nabla_{\g}\gamma \nabla\mathrlap{^*}_{\textup\g}}{\alpha}{}&=
\lambda^{-2}\normSch{\nabla_{\g}\sqrt{\gamma}}{2\alpha}{}^2
\leq 
\lambda^{-2}\|\nabla_{\g}(P+\lambda^2)^{-1/2}\|^2
\normSch{(P+\lambda^2)^{1/2}\sqrt{\gamma}}{2\alpha}{}^2\\
&\leq \lambda^{-2}\normSch{(P+\lambda^2)^{1/2}\sqrt{\gamma}}{2\alpha}{}^2\\
&\leq \lambda^{-2}\|(P+\lambda^2)^{1/2}(1+\lambda^{-1}|P-\lambda^2|)^{-1}\|^2\normSch{(1+\lambda^{-1}|P-\lambda^2|)\sqrt{\gamma}}{2\alpha}{}^2\\
&\lesssim  \normSch{(1+\lambda^{-1}|P-\lambda^2|)\sqrt{\gamma}}{2\alpha}{}^2\\
&\lesssim \normSch{\gamma}{\alpha}{}+\lambda^{-2}\normSch{(P-\lambda^2)\gamma(P-\lambda^2)}{\alpha}{}.
\end{align}
In the penultimate line we used the spectral theorem to estimate
\begin{align}\label{used spectral theorem dor gradient}
 \|(P+\lambda^2)^{1/2}(1+\lambda^{-1}|P-\lambda^2|)^{-1}\|=\sup_{\tau^2\in\spec(P)}\frac{(\tau^2+\lambda^2)^{1/2}}{1+\lambda^{-1}|\tau^2-\lambda^2|}\lesssim\lambda.   
\end{align}
\end{proof}

\section{The $C^{1,1}$ case (proof of Theorem \ref{thm:main-result_C2})}\label{sec:proof-C2-case}
In this section, we will prove Theorem \ref{thm:main-result_C2} for $q=q_n$. After a series of reductions, we show that the result follows from estimates for a regularized, frequency-localized, first-order pseudodifferential operator on $\R^n$.

\subsection{Spatial localization and regularization}\label{sect. Reduction to Euclidean space}

By Proposition \ref{prop. equivalence of cluster, qm and resolvent} and Lemma \ref{lemma triangle inequality}, the result of Theorem \ref{thm:main-result_C2} for $q=q_n$ would follow from the local bound
\begin{align}\label{local bound qm}
\normLp{\rho_{\phi\gamma\phi}}{q_n/2}{(M)}\lesssim\lambda^{2\delta(q_n)}\left(\normSch{\gamma}{\alpha(q_n)}{(L^2(M))}
    +\lambda^{-2}\normSch{(P-\lambda^2)\gamma(P-\lambda^2)}{\alpha(q_n)}{(L^2(M))}\right),
\end{align}
where $\phi$ is a bump function adapted to a coordinate chart $\Omega\subset M$, and $\delta(q_n)=1/q_n$, $\alpha(q_n)=(n+1)/n$ are as in the smooth case (see \eqref{eq-def:delta_sogge}, \eqref{eq-def:alpha}). We have dropped the spectral localization $\Pi_{\leq 2\lambda}$, which strengthens the inequality and facilitates the  localization argument. If $\psi$ is another bump function supported in $\Omega$ and such that $\psi=1$ on $\supp\phi$, then we may replace $\gamma$ by $\psi\gamma\psi$ on the right-hand side of \eqref{local bound qm}. By Lemma~\ref{lemma gradient term}, the commutator $[P,\psi]=-2\nabla_{\g}\psi\cdot\nabla_{\g}-\Delta_{\g}\psi$ may be absorbed by the two terms in theparenthesis (this requires only $C^1$ coefficients), 
\begin{align}
    &\lambda^{-2}\normSch{(P-\lambda^2)\psi\gamma\psi(P-\lambda^2)}{\alpha(q_n)}{}
    =\lambda^{-2}\normSch{(P-\lambda^2)\psi\sqrt{\gamma}}{2\alpha(q_n)}{}^2\\
    &\lesssim \lambda^{-2}\normSch{(P-\lambda^2)\sqrt{\gamma}}{2\alpha(q_n)}{}^2
+\lambda^{-2}\normSch{[P,\psi]\sqrt{\gamma}}{2\alpha(q_n)}{}^2\\
&\lesssim\normSch{\gamma}{\alpha(q_n)}{}
    +\lambda^{-2}\normSch{(P-\lambda^2)\gamma(P-\lambda^2)}{\alpha(q_n)}{}.
\end{align}
Thus, \eqref{local bound qm} may be considered as an inequality on Euclidean space, with $P$ an unbounded, non-negative self-adjoint operator on $L^2(\R^n,|\g|^{1/2}\rd x)$. Since it will be more convenient to work with operators on $L^2(\R^n):=L^2(\R^n,\rd x)$, we consider the unitary operator 
\begin{align}
   U:L^2(\R^n,|\g|^{1/2}\rd x)\to L^2(\R^n),\quad Uf(x):=|\g|^{1/4}f(x).
\end{align}
Then
\begin{align}\label{change of Hilbert space}
    UPU^{-1}u=-|\g|^{-1/4}\partial_i|\g|^{1/2}\g^{ij}\partial_j|\g|^{-1/4}u,\quad \forall u\in H^2(\R^n).
\end{align}
Since we are assuming $C^{1,1}$ regularity for the coefficients ($C^1$ would be enough), we have
$
UPU^{-1}=-\partial_i\g^{ij}\partial_j+A,
$
where $A$ is a first-order differential operator whose coefficients contain derivatives of the metric up to order one. Since Schatten norms are unitarily invariant and, again by Lemma~\ref{lemma gradient term}, the contribution from $A$ is controlled by the terms in parenthesis in \eqref{local bound qm}, we may replace $P$ by $-\partial_i\g^{ij}\partial_j$ there.

We extend the coefficients $\g^{ij}$ to all of $\R^n$ in such a way that $\g^{ij}=\delta^{ij}$ outside a ball. 
By shrinking the chart, if necessary, we may furthermore assume that
\begin{align}\label{g close to Euclidean C11}
\sup_{1\leq i,j\leq n}\|\g^{ij}-\delta^{ij}\|_{C^{1,1}(\R^n)}\ll 1. 
	\end{align}
This assumption appears in \cite{Smith-2006-C11} and facilitates some arguments. We will not use it explicitly.

By the preceding discussion, \eqref{local bound qm} would follow from the following proposition. 

\begin{proposition}\label{prop. C11 qm bound Rn}
Let $n\geq 2$. Assume that the coefficients $\g^{ij}$ are of class $C^{1,1}$, 
and let
\begin{align}\label{P=dgd}
P=-\partial_i\g^{ij}\partial_j,\quad \dom(P)=H^2(\R^n),
\end{align}
as an operator on $L^2(\R^n)$. Then
\begin{align}\label{local bound qm Rn}
\normLp{\rho_{\gamma}}{q_n/2}{(Q)}\lesssim\lambda^{2\delta(q_n)}\Big(\normSch{\gamma}{\alpha(q_n)}{(L^2(\R^n))}+\lambda^{-2}\normSch{(P-\lambda^2)\gamma (P-\lambda^2)}{\alpha(q_n)}{(L^2(\R^n))}\Big)
\end{align}
for any unit cube $Q\subset\R^n$ and any finite-rank operator $\gamma\geq 0$ with $\Ran\gamma\subset H^2(\R^n)$. The implicit constant depends on the $C^{1,1}$ norms of $\g^{ij}$ and the ellipticity constant $c_{\rm ell}$ (defined in \eqref{ellipticity g}), but is independent of $\lambda\geq 1$, $Q$, and $\gamma$.
\end{proposition}

\begin{remark}\label{remark divergence versus non-divergence form}
The reason we use $-\partial_i\g^{ij}\partial_j$ as opposed to $-\g^{ij}\partial_i\partial_j$ is that the former is self-adjoint. However, the two operators are interchangeable since \eqref{local bound qm Rn} is stable under first-order perturbations of $P$. 
\end{remark}

By the same arguments as in the proof of Proposition \ref{prop. equivalence of cluster, qm and resolvent} (see also Remark \ref{remark equivalence with A}), the bound~\eqref{local bound qm Rn} would follow from the dual resolvent estimate
\begin{align}\label{resolvent TTstar Rn}
\|W(P-(\lambda+\I)^2)^{-1}
\|_{\mathfrak{S}^{2\alpha(q_n)'}(L^2(\R^n))}^2
\lesssim \lambda^{2\delta(q_n)-2}\|W\|_{L^{2(q/2)'}(\R^n)}^2,
 \end{align}
 for all $W\in L^{2(q/2)'}(\R^n)$ supported on a unit cube $Q\subset\R^n$ and with implicit constant independent of the position of $Q$.
After a translation, we may suppose that $Q$ is centered at the origin. 
We will first replace $P$ with a regularized operator as follows. Let $\varphi\in C_c^{\infty}(\R^n)$ be supported in $B(0,2)$ and $\varphi=1$ in $B(0,1)$. Then $\varphi(D/\lambda)$ is a smooth cutoff to frequencies $|\xi| \leq  2\lambda$. Let $\g_{\sqrt{\lambda}}: = \varphi(D/\sqrt{\lambda}) \g$. Then
\begin{align}\label{eq:smoothing bounds}
	 \|\g^{ij}-\g^{ij}_{\sqrt{\lambda}}\|_{L^{\infty}(\R^n)}&\lesssim \lambda^{-1}\|\g^{ij}\|_{C^{1,1}(\R^n)},\\
     \|(\partial_i\g^{ij}-\partial_i\g_{\sqrt{\lambda}}^{ij})\|_{L^{\infty}(\R^n)}&\lesssim \lambda^{-1/2}\|\g^{ij}\|_{C^{1,1}(\R^n)}.
\end{align}
If we denote the regularized operator (i.e., the operator in \eqref{P=dgd} with $\g_{\sqrt{\lambda}}$ in place of $\g$) by $P_{\sqrt{\lambda}}$,
then 
\begin{align}
&\|W((P_{\sqrt{\lambda}}-(\lambda+\I)^2)^{-1}-(P-(\lambda+\I)^2)^{-1})
\|_{\mathfrak{S}^{2\alpha(q_n)'}}^2\\
&=\|W(P_{\sqrt{\lambda}}-(\lambda+\I)^2)^{-1}(P-P_{\sqrt{\lambda}})(P-(\lambda+\I)^2)^{-1}
\|_{\mathfrak{S}^{2\alpha(q_n)'}}^2\\
&\leq\|W(P_{\sqrt{\lambda}}-(\lambda+\I)^2)^{-1}(\g^{ij}-\g^{ij}_{\sqrt{\lambda}})\partial_i\partial_j(P-(\lambda+\I)^2)^{-1}
\|_{\mathfrak{S}^{2\alpha(q_n)'}}^2\\
&\quad+\|W(P_{\sqrt{\lambda}}-(\lambda+\I)^2)^{-1}(\partial_i\g^{ij}-\partial_i\g_{\sqrt{\lambda}}^{ij})\partial_j(P-(\lambda+\I)^2)^{-1}
\|_{\mathfrak{S}^{2\alpha(q_n)'}}^2\\
&\lesssim \lambda^{-2}\|W(P_{\sqrt{\lambda}}-(\lambda+\I)^2)^{-1}
\|_{\mathfrak{S}^{2\alpha(q_n)'}}^2\|D^2(P-(\lambda+\I)^2)^{-1})\|^2\\
&\quad+\lambda^{-1}\|W(P_{\sqrt{\lambda}}-(\lambda+\I)^2)^{-1}
\|_{\mathfrak{S}^{2\alpha(q_n)'}}^2\|D(P-(\lambda+\I)^2)^{-1})
\|^2\\
&\lesssim \|W(P_{\sqrt{\lambda}}-(\lambda+\I)^2)^{-1}
\|_{\mathfrak{S}^{2\alpha(q_n)'}}^2,
\end{align}
where $D$ is the gradient and $D^2$ denotes the matrix of second derivatives. The last inequality follows from elliptic regularity (see, e.g., \cite[Section 6.3.1]{MR2597943}) and the spectral theorem, 
\begin{align}\label{eliptic regularity D^2}
\lambda^{-1}&\|D^2(P-(\lambda+\I)^2)^{-1})\|+\lambda^{-1/2}\|D(P-(\lambda+\I)^2)^{-1})\|\\
&\lesssim \lambda^{-1}\|(P+\lambda^2)(P-(\lambda+\I)^2)^{-1})\|+\lambda^{-1/2}|(P+\lambda^2)^{1/2}(P-(\lambda+\I)^2)^{-1})\|\\
&\lesssim
\lambda^{-1}\sup_{\tau^2\in\spec(P)}\frac{\tau^2+\lambda^2}{|\tau^2-(\lambda+\I)^2|}+\lambda^{-1/2}\sup_{\tau^2\in\spec(P)}\frac{(\tau^2+\lambda^2)^{1/2}}{|\tau^2-(\lambda+\I)^2|}\lesssim 1.
\end{align}
Thus, we may replace $P$ with $P_{\sqrt{\lambda}}$ in \eqref{resolvent TTstar Rn}. We drop the subscript again and write
\begin{align}\label{Preg}
P=-\partial_i\g_{\sqrt{\lambda}}^{ij}\partial_j,\quad \dom(P)=H^2(\R^n).
\end{align}
 By Lemma \ref{lemma triangle inequality}, it suffices to prove \eqref{resolvent TTstar Rn} with $(P-(\lambda+\I)^2)^{-1}\Pi_{\leq 2\lambda}$ and $(P-(\lambda+\I)^2)^{-1}\Pi_{> 2\lambda}$ in place of $(P-(\lambda+\I)^2)^{-1}$ separately. The spectral projections $\Pi_{\leq 2\lambda}$, $\Pi_{>2\lambda}$, etc., are always understood to be associated with the current form of the operator $P$, here \eqref{Preg}.

\subsection{Elliptic estimates}\label{subsec:elliptic-est}
The following Proposition proves the part of \eqref{resolvent TTstar Rn} where the resolvent is replaced by $(P-(\lambda+\I)^2)^{-1}\Pi_{> 2\lambda}$. The Proposition is valid for all $s\in [0,2]$ (the only thing that changes in H\"older regularity is the imaginary part of the resolvent) and will also be used in later sections. 

\begin{proposition}\label{prop:elliptic}
Let $s\in(0,2]$, and let $\g$ be of class $C^s(\R^n)$. Then for $2\leq q<  \frac{2(n+(1-s)_+)}{n-2+(1-s)_+}$,
\begin{align}
 \normSch{W_1\Pi_{>2\lambda}(P-(\lambda\pm\I\mu)^2)^{-1}W_2}{(q/2)'}{(L^2(\R^n))}
 \lesssim\ \lambda^{2(n+(1-s)_+)(\frac{1}{2}-\frac{1}{q})-2}\|W_1\|_{L^{2(q/2)'}} \|W_2\|_{L^{2(q/2)'}}   
\end{align}
 for any $\lambda\geq 1$, $\mu\in\R$ and $W_1,W_2\in L^{2(q/2)'}(\R^n)$. Moreover, we have
  \begin{align}\label{elliptic W}
\|W(P-(\lambda+\I\mu)^2)^{-1}\Pi_{>2\lambda}\|_{\mathfrak{S}^{2(q/2)'}(L^2(\R^n))}^2
\lesssim \lambda^{2(n+(1-s)_+)(\frac{1}{2}-\frac{1}{q})-3-(1-s)_+}\|W\|_{L^{2(q/2)'}(\R^n)}^2
 \end{align}
 for any $W\in L^{2(q/2)'}(\R^n)$.
 The implicit constants depend on the ellipticity constant $c_{\rm ell}$ in~\eqref{ellipticity g}, but are independent of $\lambda$, $\mu$, $W_1$, $W_2$, $W$. 
\end{proposition}

\begin{remark}\label{rmk:elliptic}
    Since $\delta(q)=n(\frac{1}{2}-\frac{1}{q})-\frac{1}{2}$ for $q\geq q_n$, see \eqref{eq-def:delta_sogge}, the power of $\lambda$ for $s\geq 1$ is 
    \begin{align*}
       \lambda^{2n(\frac{1}{2}-\frac{1}{q})-2}=\lambda^{2\delta(q)-1}.
    \end{align*}
Besides, the Schatten exponent $(q_n/2)'=(n+1)/2<(n+1)=\alpha(q_n)'$ is better than needed. In particular, \eqref{elliptic W} implies, for $s\geq 1$,
\begin{align}
 \|W(P-(\lambda+\I)^2)^{-1}\Pi_{>2\lambda}\|_{\mathfrak{S}^{2\alpha(q_n)'}(L^2(\R^n)))}^2
\lesssim    \lambda^{2\delta(q_n)-2}\|W\|_{L^{2(q/2)'}(\R^n)}^2.
\end{align}
\end{remark}

\begin{proof}[Proof of Proposition \ref{prop:elliptic}]
Since the operators $P$ and $\Pi_{>2\lambda}$ commute, H\"older's inequality implies
\begin{align}
\normSch{W_1\Pi_{>2\lambda}(P-(\lambda\pm\I\mu)^2)^{-1}W_2}{(q/2)'}{}
\leq& \normSch{W_1\Pi_{>2\lambda}|P-(\lambda\pm\I\mu)^2|^{-1/2}}{2(q/2)'}{}\\
&\times \normSch{\Pi_{>2\lambda}|P-(\lambda\pm\I\mu)^2|^{-1/2}W_2}{2(q/2)'}{}\\
\end{align}
By Proposition \ref{Prop. KSS low reg.}, applied with $p=2(q/2)'$ and $\beta(\tau)=\abs{\tau^2-(\lambda\pm\I\mu)^2}^{-1/2}$, and by Lemma~\ref{lemma sums a,b}, applied with $a=(q/2)'$ and $b=n-1+(1-s)_+$, it follows that for any $2\leq q< \frac{2(n+(1-s)_+)}{n-2+(1-s)_+} $, we have
\begin{align}
    \normSch{W_1\Pi_{>2\lambda}|P-(\lambda\pm\I\mu)^2|^{-1/2}}{2(q/2)'}{}
    &\lesssim \|W_1\|_{L^{2(q/2)'}}\left(\sum_{k>2\lambda}(k^2+\lambda^2)^{-(q/2)'}k^{n-1+(1-s)_+}\right)^{\frac{1}{2(q/2)'}}\\
    &\lesssim \lambda^{(n+(1-s)_+)(\frac{1}{2}-\frac{1}{q})-1}\|W_1\|_{L^{2(q/2)'}},
\end{align}
and similarly for the factor involving $W_2$. 
The second claim follows from 
 \begin{align}
&\|W(P-(\lambda+\I\mu)^2)^{-1}\Pi_{>2\lambda}\|_{\mathfrak{S}^{2(q/2)'}}^2\\ 
&=\|W(P-(\lambda+\I\mu)^2)^{-1}\Pi_{>2\lambda}(P-(\lambda-\I\mu)^2)^{-1}\overline{W}\|_{\mathfrak{S}^{(q/2)'}}\\
&=2^{-1}\lambda^{-1-(1-s)_+}\|W\Pi_{>2\lambda}((P-(\lambda+\I\mu)^2)^{-1}-(P-(\lambda-\I\mu)^2)^{-1})\overline{W}\|_{\mathfrak{S}^{(q/2)'}}\\
&\leq 2^{-1}\lambda^{-1-(1-s)_+}\sum_{\pm}\|W\Pi_{>2\lambda}(P-(\lambda\pm\I\mu)^2\overline{W}\|_{\mathfrak{S}^{(q/2)'}}
\\
&\lesssim 
\lambda^{2(n+(1-s)_+)(\frac{1}{2}-\frac{1}{q})-3-(1-s)_+}
\|W\|_{L^{2(q/2)'}}^2,
 \end{align}
 where we used the resolvent identity in the third line.
\end{proof}

\begin{remark}\label{remark KSS valid for Rn}
 Since the proof of Proposition \ref{Prop. KSS low reg.} only uses the $L^2(M)\to L^{\infty}(M)$ bounds~\eqref{Linfty bounds low regularity}, the assumption that $M$ is compact can be dropped. Thus, the results of Proposition~\ref{Prop. KSS low reg.} and Proposition~\ref{prop:elliptic} also hold for variable-coefficient second-order elliptic operators on $\R^n$.
\end{remark}

\subsection{Microlocalization}\label{subsec. Frequency-localized estimates}
To complete the proof of Proposition \ref{prop. C11 qm bound Rn}, it remains to prove the spectrally localized estimate
\begin{align}\label{spectrally localized TTstar resolvent bound}
\|W\Pi_{\leq2\lambda}(P-(\lambda+\I)^2)^{-1}
\|_{\mathfrak{S}^{2\alpha(q_n)'}(L^2(\R^n))}^2
\lesssim \lambda^{2\delta(q_n)-2}\|W\|_{L^{2(q/2)'}(\R^n)}^2,
 \end{align}
for all $W\in L^{2(q/2)'}(\R^n)$ supported on a unit cube $Q\subset\R^n$ and for $P$ as in \eqref{Preg}. We will use pseudodifferential calculus to reduce \eqref{spectrally localized TTstar resolvent bound} to a bound for a first-order pseudodifferential operator. This will allow us to use the results of Koch and Tataru \cite[Section 4.1]{KochTataru-2005}, in particular the dispersive estimate \cite[Proposition 4.7]{KochTataru-2005}.

We start with a reduction to a frequency-localized estimate.

\begin{lemma}\label{lemma insert req. loc.}
Let $\phi$ be a bump function adapted to $Q$, and let $\varphi$ be a bump function adapted to the unit ball $B(0,1)$. Then there exists $c>0$ such that   
 \begin{align}
 \|W\Pi_{\leq2\lambda}(P-(\lambda+\I)^2)^{-1}
\|_{\mathfrak{S}^{2\alpha(q_n)'}}^2
 \leq\|W\phi\varphi(c\lambda^{-1}D)(P-(\lambda+\I)^2)^{-1}
\|_{\mathfrak{S}^{2\alpha(q_n)'}}^2
 \end{align}
 for all $W\in L^{2(q/2)'}(\R^n)$ supported in $Q$.
\end{lemma}

\begin{proof}
If $f\in H^2(\R^n)$ is such that $\supp(\widehat{f})\subset \{|\xi|> c^{-1}\lambda\}$, then by \eqref{ellipticity g} and \eqref{eq:smoothing bounds}, 
 \begin{align}
     \langle f,Pf\rangle
     &\geq (c_{\rm ell}^{-1}-C\lambda^{-1})\|\nabla f\|^2\\
     &\geq \lambda^2(c^{-2}c_{\rm ell}^{-1}-C\lambda^{-1})\|f\|^2.
 \end{align}
 Hence, if we choose $c^2\leq (4c_{\rm ell})^{-1}$ and if $\lambda\gg 1$, then 
 \begin{align}
     \langle f,Pf\rangle\geq \lambda^2c^{-2}(2c_{\rm ell})^{-1}\|f\|^2.
 \end{align}
 If we further require $c$ to be so small that $c^{-2}(2c_{\rm ell})^{-1}>4$, then it follows that $\Pi_{\leq2\lambda}f=0$. 
 In other words, for such a choice of $c$, we have 
 \begin{align}\label{Pileq2lambda replaced by Fourier multiplier}
\Pi_{\leq2\lambda}=\Pi_{\leq2\lambda}\mathbf{1}(|D|\leq c^{-1}\lambda).    
 \end{align}
By taking adjoints, we also find that
\begin{align}\label{Pileq2lambda replaced by Fourier multiplier 2}
\Pi_{\leq2\lambda}=\mathbf{1}(|D|\leq c^{-1}\lambda)\Pi_{\leq2\lambda}.    
 \end{align}
 Therefore,
 \begin{align}
 \|W\Pi_{\leq2\lambda}(P-(\lambda+\I)^2)^{-1}
\|_{\mathfrak{S}^{2\alpha(q_n)'}}^2
 &\leq  \|W\phi\Pi_{\leq2\lambda}(P-(\lambda+\I)^2)^{-1}
\|_{\mathfrak{S}^{2\alpha(q_n)'}}^2\\
&=\|W\phi\varphi(c\lambda^{-1}D)\Pi_{\leq2\lambda}(P-(\lambda+\I)^2)^{-1}
\|_{\mathfrak{S}^{2\alpha(q_n)'}}^2\\
&\leq\|W\phi\varphi(c\lambda^{-1}D)(P-(\lambda+\I)^2)^{-1}
\|_{\mathfrak{S}^{2\alpha(q_n)'}}^2.
 \end{align}
\end{proof}
Thus, \eqref{spectrally localized TTstar resolvent bound} would follow from the following bound,
\begin{align}\label{spectrally localized TTstar resolvent bound smooth cutoffs}
\|W\phi\varphi(c\lambda^{-1}D)(P-(\lambda+\I)^2)^{-1}\Pi_{\leq2\lambda}\|_{\mathfrak{S}^{2\alpha(q_n)'}(L^2(\R^n))}^2
\lesssim \lambda^{2\delta(q_n)-2}\|W\|_{L^{2(q_n/2)'}(\R^n)}^2,
 \end{align}
 for all $W\in L^{2(q/2)'}(\R^n)$.

\begin{remark}\label{remark lemma insert req. loc.}
The proof of Lemma \ref{lemma insert req. loc.} only uses ellipticity of $P$. Thus, the result remains true for coefficients of regularity $C^s$, for any $s\in [0,2]$.
\end{remark}

Next, we change quantization to align with the set-up in \cite{KochTataru-2005}.

\begin{lemma}\label{lemma change of quantization}
Under the ame assumptions of Lemma \ref{lemma insert req. loc.}, there exists $\chi_0\in C^{\infty}_c(T^*\R^n)$ such that for any $N\geq 0$
\begin{align}
&\|W\phi\varphi(c\lambda^{-1}D)(P-(\lambda+\I)^2)^{-1}\Pi_{\leq2\lambda}
\|_{\mathfrak{S}^{2\alpha(q_n)'}(L^2(\R^n))}^2 \\
&\lesssim_N \|W\chi_0^w(x,\lambda^{-1}D)(P-(\lambda+\I)^2)^{-1}\|_{\mathfrak{S}^{2\alpha(q_n)'}(L^2(\R^n))}^2
+\lambda^{-N}\|W\|_{L^{2(q_n/2)'}(\R^n)}^2.
\end{align}
\end{lemma}

\begin{proof}
The operator $\phi\varphi(c\lambda^{-1}D)$ in~\eqref{spectrally localized TTstar resolvent bound smooth cutoffs} is the Kohn--Nirenberg (left) quantization of the symbol $\phi(x)\varphi(c\lambda^{-1}\xi)$. 
By the change of quantization formula (see, e.g., \cite[Section 6.1]{MR1872698}), there exists $\chi_0\in C^{\infty}_c(T^*\R^n)$ such that $\phi\varphi(c\lambda^{-1}D)=\chi_0^w(x,\lambda^{-1}D)+R$, where $R\in OPS^{-\infty}_{1,0}$ is a smoothing operator, thus by \eqref{smoothing operator trace class},
\begin{align}
\|WR(P-(\lambda+\I)^2)^{-1}
\|_{\mathfrak{S}^{2\alpha(q_n)'}}^2
\leq \|WR\|_{\mathfrak{S}^{2\alpha(q_n)'}}^2\|(P-(\lambda+\I)^2)^{-1}
\|^2
\lesssim \lambda^{-N}\|W\|_{L^{2(q_n/2)'}(\R^n)}^2.
\end{align}
Since $W$ is supported in $Q$, the claim follows from H\"older's inequality.
\end{proof}

 By shrinking $Q$, we may assume in the following that $W$ is supported in the unit ball $|x|\leq 1$. We recall the definition
\begin{align}
    B_{\lambda}:=\{(x,\xi)\in T^*\R^n:|x|\leq 1,|\xi|\leq \lambda\}.
\end{align}
If we set $\chi(x,\xi):=\chi_0(x,\lambda^{-1}\xi)$, then clearly $\chi\in S_{\lambda,\lambda^{1/2}}$, and we may furthermore assume that $\chi$ is supported in $B_{2c^{-1}\lambda}$. By Lemma \ref{lemma change of quantization}, it suffices to prove
\begin{align}\label{microlocalized TTstar resolvent bound}
\|W\chi^w(x,D)(P-(\lambda+\I)^2)^{-1}\|_{\mathfrak{S}^{2\alpha(q_n)'}(L^2(\R^n))}^2
\lesssim \lambda^{2\delta(q_n)-2}\|W\|_{L^{2(q_n/2)'}(\R^n)}^2.
 \end{align}

The principal symbol of $P$ is given by 
$p(x,\xi)=\g_{\sqrt{\lambda}}^{ij}(x)\xi_i\xi_j$. We smoothly modify $p$ outside $|\xi|\leq 2c^{-1}\lambda$ such that $p(x,\xi)=2\lambda^2$ for $|\xi|\geq 3c^{-1}\lambda$.
It then follows from the definition of $\g_{\sqrt{\lambda}}$ that $p\in \lambda^2C^2S_{\lambda,\lambda^{1/2}}$. 
To put ourselves in the setting of \cite{KochTataru-2005}, we multiply \eqref{microlocalized TTstar resolvent bound} by $\lambda^2$ and redefine
\begin{align}
    p(x,\xi):=\lambda^{-1}(\g_{\sqrt{\lambda}}^{ij}(x)\xi_i\xi_j-\lambda^2)
\end{align}
for $|\xi|\leq 2c^{-1}\lambda$, and $p(x,\xi):=\lambda$ for $|\xi|\geq 3c^{-1}\lambda$. Then $p\in \lambda C^2S_{\lambda,\lambda^{1/2}}$ satisfies the assumptions of \cite[Theorem 2.5]{KochTataru-2005}. Specifically,
\begin{itemize}
    \item[(A1)] $p$ is real-valued and  of principal type, i.e., 
    \begin{align}
        |\partial_{\xi}\,p(x,\xi)|\gtrsim 1,\quad\mbox{in }\Sigma:=\{(x,\xi)\in T^*\R^n:p(x,\xi)=0\}.
    \end{align}
    \item[(A2)] The fibers $\Sigma_x:=\{\xi\in T^*_x\R^n:p(x,\xi)=0\}$ have everywhere non-vanishing Gaussian curvature. More precisely, for the second fundamental form $\mathbf{II}$ of $\Sigma_x$, we have \[|\det \mathbf{II}|\gtrsim \lambda^{1-n}.\]
\end{itemize}
The lower bound in (A2) is natural since the second fundamental form of a symbol in $\lambda C^2S_{\lambda,\lambda^{1/2}}$ has size $\lambda^{-1}$. Koch and Tataru also consider the case when $\Sigma_x$ has only $n-1-k$ nonvanishing curvatures, with $0\leq k\leq n-2$, but we will only need $k=0$ here.

\begin{proposition}\label{prop:qm-bound-KT05}
 Let $p\in \lambda C^2S_{\lambda,\lambda^{1/2}}$ satisfies assumptions {\rm(A1)} and {\rm(A2)} above, and let $\chi\in S_{\lambda,\lambda^{1/2}}$ be supported in $B_{\lambda}$. Then
 \begin{align}\label{KT05 TTstar resolvent bound}
    \|W\chi^w(x,D)(p^w(x,D)\mp 2\I)^{-1}\|_{\mathfrak{S}^{2\alpha(q_n)'}(L^2(\R^n))}
\lesssim \lambda^{\delta(q_n)}\|W\|_{L^{2(q_n/2)'}(\R^n)}
\end{align}
for all $W\in L^{2(q_n/2)'}(\R^n)$ supported in $B(0,1)$.
\end{proposition}

We postpone the proof of Proposition \ref{prop:qm-bound-KT05} to the next subsection. In the remainder of this subsection, we prove that it implies Proposition \ref{prop. C11 qm bound Rn}.
First, we show that \eqref{KT05 TTstar resolvent bound} holds for $\chi$ supported in $B_{2c^{-1}\lambda}$. To this end, we consider a finite covering 
\begin{align}
B_{2c^{-1}\lambda}\subset\bigcup_{j=1}^K B_{\lambda}^j,\quad B_{\lambda}^j:=\{(x,\xi)\in T^*\R^n:|x|\leq 1,|\xi-\xi_j|\leq \lambda\},
\end{align}
and a corresponding partition of unity $\{\chi_j\}_{j=1}^K$.
The assumptions (A1) and (A2) are satisfied for $p(x+x_0,\xi+\xi_0)$, uniformly in $(x_0,\xi_0)\in T^*\R^n$. Applying \eqref{KT05 TTstar resolvent bound} with $\chi_j$ in place of $\chi$ and using the triangle inequality, the claim follows.

\begin{proof}[Proof of Proposition \ref{prop. C11 qm bound Rn}]
It remains to prove \eqref{KT05 TTstar resolvent bound} $\implies$ \eqref{microlocalized TTstar resolvent bound}, which would follow from
\begin{align}\label{end of proof of Proposition C11}
\|W\chi^w((p^w- 2\I)^{-1}-\lambda(P-(\lambda+\I)^2)^{-1})\|_{\mathfrak{S}^{2\alpha(q_n)'}}\lesssim \lambda^{\delta(q_n)}\|W\|_{L^{2(q_n/2)'}}.
\end{align}
By the resolvent identity and the triangle inequality,
\begin{align}
\|W\chi^w&((p^w- 2\I)^{-1}-\lambda(P-(\lambda+\I)^2)^{-1})\|_{\mathfrak{S}^{2\alpha(q_n)'}}\\
&=\|W\chi^w(p^w- 2\I)^{-1}(p^w-\lambda^{-1}(P-\lambda^2+1))\lambda(P-(\lambda+\I)^2)^{-1}\|_{\mathfrak{S}^{2\alpha(q_n)'}}\\    
&\leq \|W(p^w- 2\I)^{-1}\chi^w(p^w-\lambda^{-1}(P-\lambda^2+1))\lambda(P-(\lambda+\I)^2)^{-1}\|_{\mathfrak{S}^{2\alpha(q_n)'}}\\
&\quad+\|W[\chi^w,(p^w- 2\I)^{-1}](p^w-\lambda^{-1}(P-\lambda^2+1))\lambda(P-(\lambda+\I)^2)^{-1}\|_{\mathfrak{S}^{2\alpha(q_n)'}}.
\end{align}
We observe that
\begin{align}
\chi^w(p^w-\lambda^{-1}(P-\lambda^2+1)) \in  OPS_{\lambda,\lambda^{1/2}}
\end{align}
(since its principal symbol vanishes) and $[p^w,\chi^w]\in OPS_{\lambda,\lambda^{1/2}}$. Thus, these operators are $L^2$ bounded, uniformly in $\lambda$. Then, 
\begin{align}
\|W(p^w- 2\I)^{-1}&\chi^w(p^w-\lambda^{-1}(P-\lambda^2+1))\lambda(P-(\lambda+\I)^2)^{-1}\|_{\mathfrak{S}^{2\alpha(q_n)'}}\\
&\leq \|W(p^w- 2\I)^{-1}\|_{\mathfrak{S}^{2\alpha(q_n)'}}\|\chi^w(p^w-\lambda^{-1}(P-\lambda^2+1))\|\|\lambda(P-(\lambda+\I)^2)^{-1}\|\\
&\lesssim \|W(p^w- 2\I)^{-1}\|_{\mathfrak{S}^{2\alpha(q_n)'}}\\
&\lesssim \lambda^{\delta(q_n)}\|W\|_{L^{2(q_n/2)'}},
\end{align}
and similarly,
\begin{align}
&\|W[\chi^w,(p^w- 2\I)^{-1}](p^w-\lambda^{-1}(P-\lambda^2+1))\lambda(P-(\lambda+\I)^2)^{-1}\|_{\mathfrak{S}^{2\alpha(q_n)'}}\\
&=\|W(p^w- 2\I)^{-1}[p^w,\chi^w](p^w- 2\I)^{-1}(p^w-\lambda^{-1}(P-\lambda^2+1))\lambda(P-(\lambda+\I)^2)^{-1}\|_{\mathfrak{S}^{2\alpha(q_n)'}}\\
&\leq \|W(p^w- 2\I)^{-1}\|_{\mathfrak{S}^{2\alpha(q_n)'}}
\|[p^w,\chi^w]\|\|(p^w- 2\I)^{-1}(p^w-\lambda^{-1}(P-\lambda^2+1))\lambda(P-(\lambda+\I)^2)^{-1}\|\\
&\lesssim \|W(p^w- 2\I)^{-1}\|_{\mathfrak{S}^{2\alpha(q_n)'}}\\
&\lesssim \lambda^{\delta(q_n)}\|W\|_{L^{2(q_n/2)'}},
\end{align}
where we used \eqref{KT05 TTstar resolvent bound} in the final inequalities of the last two displays.

\end{proof}

\subsection{Factorization}\label{subs. factorization}

We continue with the reduction of \eqref{KT05 TTstar resolvent bound}.
By a partition of unity argument, we may assume that $\chi$ localizes in $\xi$ to the conic region $\Gamma=\{|\xi'|\ll \xi_1\approx\lambda\}$, where $\xi=(\xi_1,\xi')\in\R\times\R^{n-1}$.  
We will restrict our attention to an $\eps\lambda$-neighborhood of the characteristic set $\Sigma$, where $\eps>0$ is sufficiently small but fixed. For the complement of this set (in the elliptic region), \eqref{microlocalized TTstar resolvent bound} follows again from Proposition~\ref{prop:elliptic}.

We thus assume that $\Gamma$ is contained in an $\eps\lambda$-neighborhood of $\Sigma$, and $|\xi'|\ll \xi_1\approx\lambda$ in $\Gamma$.
By Assumption (A1) and the implicit function theorem, there exist locally defined real-valued smooth functions $a,e$ such that
\begin{align}\label{factorization}
    p(x,\xi)=e(x,\xi)(\xi_1-a(x,\xi',\lambda))\quad \forall (x,\xi)\in\Gamma.
\end{align}
By Assumption (A2), it follows that $\partial_{\xi'}^2a(x,\xi',\lambda)$ is non-degenerate.

We extend $a$ and $e$ globally to symbols in $\lambda C^2S_{\lambda,\lambda^{1/2}}$ and $S_{\lambda,\lambda^{1/2}}$, respectively, with $e$ elliptic. Then $e^w$ has a parametrix 
in $OPS_{\lambda,\lambda^{1/2}}$,
and is thus invertible in $L^2$ for $\lambda\gg 1$ with norm $\mathcal{O}(1)$. By pseudodifferential calculus,
\begin{align}\label{symbol factorization}
  p^w\chi^w=e^w(D_1-a^w)\chi^w+R,
\end{align}
where $R\in OPS_{\lambda,\lambda^{1/2}}$.  
We will show that \eqref{KT05 TTstar resolvent bound}
would follow from the following lemma. 

\begin{lemma}\label{lemma KT05 type}
Let $a\in \lambda C^2S_{\lambda,\lambda^{1/2}}$. Assume that $\partial_{\xi'}^2a(x,\xi',\lambda)$ is non-degenerate, in the sense that
\begin{align}
    |\det \partial_{\xi'}^2a(x,\xi',\lambda)|\gtrsim \lambda^{1-n}.
\end{align}
Let $\chi=\chi(x,\xi')$ be a bump function supported in $B_{\lambda}':=\{(x,\xi'):|x|\leq 1,\,|\xi'|\leq \lambda\}$ and smooth on this scale. Then
\begin{align}
\|W_1(D_1-a^w\mp\I)^{-1}\chi^wW_2\|_{\mathfrak{S}^{\alpha(q_n)'}(L^2(\R^n))} \lesssim \lambda^{2\delta(q_n)}\|W_1\|_{L^{2(q_n/2)'}(\R^n)}\|W_2\|_{L^{2(q_n/2)'}(\R^n)}
\end{align}
for all $W_1,W_2\in L^{2(q_n/2)'}(\R^n)$ supported in the unit ball of $\R^n$.
\end{lemma}

\begin{remark}\label{remark chi weaker asspt.}
The assumption on $\chi$ in Lemma \ref{lemma KT05 type} is weaker than in Proposition \ref{prop:qm-bound-KT05} since $B_{\lambda}\subset B_{\lambda}'$. 
\end{remark}

For the proof of Lemma \ref{lemma KT05 type}, we will use
the multilinear Hardy-Littlewood-Sobolev inequality, recalled below.
\begin{proposition}\label{HLS multilin}
    Assume that $(\beta_{ij})_{1\leq i,j\leq N}$ and $(r_k)_{1\leq k\leq N}$ are real numbers such that for all $1\leq i, j, k\leq N$
    \begin{align*}
    \beta_{ii}=0, \quad 0\leq \beta_{ij} =\beta_{ji}< 1,\quad r_k>1,\quad \sum_{k=1}^N\frac{1}{r_k}>1,\quad \sum_{i=1}^N \beta_{ik}=\frac{2(r_k-1)}{r_k} .
    \end{align*}
    Then there exists $C>0$ such that 
    \[\left| \int_\R\ldots\int_\R \frac{f_1(t_1)\dots f_N(t_N)}{\prod_{i<j}\abs{t_i-t_j}^{\beta_{ij}}} dt_1\dots dt_N \right| \leq C \prod_{k=1}^N\normLp{f_k}{r_k}{(\R)} \]
    for all $1\leq k\leq N$, $f_k \in L^{r_k}(\R)$.
\end{proposition}

\begin{proof}[Proof of Lemma \ref{lemma KT05 type}]
As in \cite{KochTataru-2005}, we will denote coordinates by $(t,x)\in\R\times\R^{n-1}$ instead of $(x_1,x')$ in the following. 
 Let $S(t,r):L^2(\R^{n-1})\to L^2(\R^{n-1})$ denote the unitary propagator associated to the family of bounded self-adjoint operators $a^w(t,x',D')$ on $L^2(\R^{n-1})$, parametrized by $t$. It is the unique solution to the Cauchy problem
\begin{equation}
\begin{cases*}
(D_t-a^w(t,x',D'))S(t,r) =0, \\
S(r,r)=\mathbf{1}.
\end{cases*}
\end{equation}
Identifying $L^2(\R^n)$ with $L^2(\R,L^2(\R^{n-1}))$, we have, for $f\in L^2(\R,L^2(\R^{n-1}))$,
\begin{align}
(D_t-a^w(t,x',D')-\I)^{-1}f(t)&=\I\int_{-\infty}^t\e^{-(t-r)}S(t,r)f(r)\rd r,\label{variation of parameters}\\
(D_t-a^w(t,x',D')+\I)^{-1}f(t)&=-\I\int_{t}^{\infty}\e^{(t-r)}S(t,r)f(r)\rd r.\label{variation of parameters 2}
\end{align}
We will prove the estimate in Lemma \ref{lemma KT05 type} for \eqref{variation of parameters}; the proof of \eqref{variation of parameters 2} is analogous.

Koch and Tataru \cite[Proposition 4.7]{KochTataru-2005} proved the following dispersive estimate,
\begin{align}\label{dispersive estimate Koch-Tataru}
    \|S(t,r)\chi^w\|_{L^1(\R^{n-1})\to L^{\infty}(\R^{n-1})}\lesssim |t-r|^{-\frac{n-1}{2}}\lambda^{\frac{n-1}{2}},\quad |t-r|\leq \eps,
\end{align}
for some $\eps>0$. 
By the Dunford--Pettis theorem \cite[Theorem 2.2.5]{Dunford-Pettis-1940-Linear} the operator 
\begin{align}
    K(t,r):=\e^{-(t-r)}\mathbf{1}_{r\leq t}S(t,r)\chi^w
\end{align}
has a kernel satisfying the bound
\begin{align}\label{Dunford--Pettis}
    \|K(t,r)\|_{L^{\infty}(\R^{n-1}\times \R^{n-1})}
    \lesssim |t-r|^{-\frac{n-1}{2}}\lambda^{\frac{n-1}{2}},\quad |t-r|\leq \eps.
\end{align}
We denoted both the operator and the kernel by $K(t,r)$. Assuming, as we may, that $W$ is supported in a ball of radius $\eps/2$, then \eqref{Dunford--Pettis} yields the Hilbert--Schmidt bound
\begin{align}\label{Hilbert--Schmidt}
&\|W(t,\cdot)K(t,r)W(r,\cdot)\|_{\mathfrak{S}^{2}(L^2(\R^{n-1}))}   
 \\&\qquad
 \lesssim |t-r|^{-\frac{n-1}{2}}\lambda^{\frac{n-1}{2}}\|W(t,\cdot)\|_{L^2(\R^{n-1})}\|W(r,\cdot)\|_{L^2(\R^{n-1})}.
\end{align}
Unitarity of $S(t,r)$ implies the operator norm bound
\begin{align}\label{operator norm}
 \|W(t,\cdot)K(t,r)W(r,\cdot)\|_{L^2(\R^{n-1})\to L^2(\R^{n-1})}   
 \lesssim \|W(t,\cdot)\|_{L^{\infty}(\R^{n-1})}\|W(r,\cdot)\|_{L^{\infty}(\R^{n-1})}.
\end{align}
Using the analytic operator family $W(t,\cdot)^zK(t,r)W(r,\cdot)^z$, complex interpolation between \eqref{Hilbert--Schmidt} and \eqref{operator norm} then yields 
\begin{align}\label{complex interpolation}
 \|W(t,\cdot)&K(t,r)W(r,\cdot)\|_{\mathfrak{S}^{n+1}(L^2(\R^{n-1}))}\\   
 &\lesssim \lambda^{\frac{n-1}{n+1}}|t-r|^{\frac{n-1}{n+1}}\|W(t,\cdot)\|_{L^{n+1}(\R^{n-1})}\|W(r,\cdot)\|_{L^{n+1}(\R^{n-1})}.
\end{align}
This, together with \eqref{variation of parameters} and Proposition \ref{HLS multilin}, yields
\begin{align}
 &\normSch{W_1(D_t-a^w-\I)^{-1}\chi^wW_2}{n+1}{}^{n+1}=\tr_{L^2(\R^n)}[W_1(D_t-a^w-\I)^{-1}\chi^wW_2]^{n+1}\\
&=\int_{\R}\ldots \int_{\R}\tr_{L^2(\R^{n-1})}[W(t_1,\cdot)K(t_1,t_2)W(t_2,\cdot)^2K(t_2,t_3)W(t_3,\cdot)^2\ldots\\
&\qquad\qquad\quad\ldots W(t_{n+1},\cdot)^2 K(t_{n+1},t_1)W(t_1,\cdot)]\rd t_1\ldots \rd t_{n+1}\\
&\leq \int_{\R}\ldots \int_{\R}\|W(t_j,\cdot)K(t_j,t_{j+1})W(t_{j+1},\cdot)\|_{\mathfrak{S}^{n+1}(L^2(\R^{n-1}))}  \rd t_1\ldots \rd t_{n+1}
\\&\lesssim
\lambda^{n-1}\int_{\R}\ldots \int_{\R}\prod_{j=1}^{n+1}\frac{\|W(t_j,\cdot)\|_{L^{n+1}(\R^{n-1})}\|W(t_{j+1},\cdot)\|_{L^{n+1}(\R^{n-1})}}{|t_j-t_{j+1}|^{\frac{n-1}{n+1}}} \rd t_1\ldots \rd t_{n+1}
\\&\lesssim \lambda^{n-1}\|W\|_{L^{n+1}(\R^{n})}^{2(n+1)}.
\end{align}
We used Proposition \ref{HLS multilin} with $t_{n+2}:=t_1$, $N=n+1$, $r_k=\frac{n+1}{2}$, $\beta_{i,i+1}=\beta_{i+1,i}=\frac{n-1}{n+1}$ and all other $\beta_{ij}=0$. 
\end{proof}

\begin{proof}[Proof of Proposition \ref{prop:qm-bound-KT05}]
Using \eqref{symbol factorization}, a straighforward calculation shows that
\begin{align}
    \chi^w (p^w\mp 2\I)^{-1}=(D_1-a^w-\I)^{-1}(e^w)^{-1}\chi^w+(D_1-a^w-\I)^{-1}R_{\pm}(p^w\mp 2\I)^{-1},
\end{align}
where
\begin{align}
     R_{\pm}=(e^w)^{-1}([p^w,\chi^w]\pm 2\I\chi^w-R)-\I\chi^w\in C^2OPS_{\lambda,\lambda^{1/2}}.
\end{align}
If $\tilde{\chi}\in S_{\lambda,\lambda^{1/2}}$ is supported in $B_{\lambda}$ and such that $\tilde{\chi}=1$ on $\supp\chi$, then
\begin{align}
 \chi^w (p^w\mp 2\I)^{-1}&=\tilde{\chi}^w(D_1-a^w-\I)^{-1}(e^w)^{-1}\chi^w
 +\tilde{\chi}^w(D_1-a^w-\I)^{-1}R_{\pm}(p^w\mp 2\I)^{-1}\\
&\quad +(1-\tilde{\chi}^w)\chi^w(p^w\mp 2\I)^{-1}.
\end{align}
Since $(1-\tilde{\chi}^w)\chi^w$ is a smoothing operator, we have
\begin{align}
 \|W(1-\tilde{\chi}^w)\chi^w(p^w\mp 2\I)^{-1}\|_{\mathfrak{S}^{2\alpha(q_n)'}} \lesssim \|W(1-\tilde{\chi}^w)\chi^w\|_{\mathfrak{S}^{2}}   \lesssim_N \lambda^{-N}\|W\|_{L^2}\lesssim \lambda^{-N}\|W\|_{L^{2(q_n/2)'}}
\end{align}
for any $N>0$, where the last inequality follows by H\"older's inequality and the compact support assumption on $W$.
We conclude that 
\begin{align*}
 \|W\chi^w(p^w\mp 2\I)^{-1}\|_{\mathfrak{S}^{2\alpha(q_n)'}} \lesssim_N 
 \|W\tilde{\chi}^w(D_1-a^w-\I)^{-1}\|_{\mathfrak{S}^{2\alpha(q_n)'}}+\lambda^{-N}\|W\|_{L^{2(q_n/2)'}}.
\end{align*}
By the resolvent identity and Lemma \ref{lemma KT05 type}, we have
\begin{align}
\|W\tilde{\chi}^w(D_1-a^w-\I)^{-1}\|_{\mathfrak{S}^{2\alpha(q_n)'}}^2&=\|W\tilde{\chi}^w(D_1-a^w-\I)^{-1}(D_1-a^w+\I)^{-1}\tilde{\chi}^w\overline{W}\|_{\mathfrak{S}^{\alpha(q_n)'}}\\
&=2\sum_{\pm}\|W\tilde{\chi}^w(D_1-a^w\mp\I)^{-1}\tilde{\chi}^w\overline{W}\|_{\mathfrak{S}^{\alpha(q_n)'}}\\
&\lesssim \lambda^{2\delta(q_n)}\|W\|_{L^{2(q_n/2)'}}^2.
\end{align}
We have used Remark \ref{remark chi weaker asspt.} and that \eqref{Dunford--Pettis} also holds for $\chi^wK(t,r)$ since
\begin{align}
    \|\chi^wS(t,r)\chi^w\|_{L^{1}(\R^{n-1})\to L^{\infty}(\R^{n-1})}\leq \|\chi^w\|_{L^{\infty}(\R^{n-1})\to L^{\infty}(\R^{n-1})}\|S(t,r)\chi^w \|_{L^{1}(\R^{n-1})\to L^{\infty}(\R^{n-1})}
\end{align}
and $\|\chi^w\|_{L^{\infty}\to L^{\infty}}\lesssim 1$ in view of the kernel bound
\begin{align}
 |\chi^w(x,y)|\leq C_N \lambda^{n-1}(1+\lambda|x-y|)^{-N}\quad\forall N>0,\quad\forall x,y\in \R^{n-1}   
\end{align}
and Young's inequality.
\end{proof}

\section{The case $s\in[1,2)$ (proof of Theorem \ref{thm:main-result_Lip})}\label{sec:proof-Lip-case}
In this Section, we prove Theorem \ref{thm:main-result_Lip}. 
It suffices to prove~\eqref{eq:spectral cluster bounds} for $q=q_n$. This would follow from the analogue of \eqref{local bound qm} with $\delta_s(q_n)$ in place of $\delta(q_n)$. The reductions in Subsection \ref{sect. Reduction to Euclidean space} leading to Proposition \ref{prop. C11 qm bound Rn} only require Lipschitz coefficients\footnote{More precisely, $C^1$ (continuously differentiable) rather than Lipschitz, but we shall not make a distinction between the two. The reductions in the Lipschitz case follow from the quadratic form methods in Section \ref{sec:proof-Holder-case}.}, as does the elliptic regularity estimate \eqref{eliptic regularity D^2}. It is thus sufficient to prove the following analogue of Proposition \ref{prop. C11 qm bound Rn}. 

\begin{proposition}\label{prop. Lipschitz qm bound Rn}
Let $n\geq 2$ and $s\in [1,2)$. Assume that the coefficients $\g^{ij}$ are of class $C^{s}$. 
Let $P$ be given by \eqref{P=dgd}.
Then
\begin{align}\label{Lipschitz qm bound Rn}
\normLp{\rho_{\gamma}}{q_n/2}{(Q)}\lesssim\lambda^{2\delta_s(q_n)}(\normSch{\gamma}{\alpha(q_n)}{(L^2(\R^n))}+\lambda^{-2}\normSch{(P-\lambda^2)\gamma (P-\lambda^2)}{\alpha(q_n)}{(L^2(\R^n))})
\end{align}
for any unit cube $Q\subset\R^n$ and any finite-rank operator $\gamma\geq 0$ with $\Ran\gamma\subset H^2(\R^n)$. The implicit constant depends on the $C^{s}$ norms of $\g^{ij}$ and the ellipticity constant $c_{\rm ell}$, but is independent of $\lambda\geq 1$, $Q$, and $\gamma$.
\end{proposition}

\subsection{Regularization and rescaling}
We regularize the coefficients at the coarser scale $\lambda^{\frac{2}{2+s}}$ compared to the $C^{1,1}$ case, i.e., we set $\g_{s}: = \varphi(\lambda^{-\frac{2}{2+s}}D) \g$. Then
\begin{align}\label{eq:smoothing bounds 2}
	 \|\g^{ij}-\g^{ij}_{s}\|_{L^{\infty}(\R^n)}&\lesssim \lambda^{-\frac{2s}{2+s}}\|\g^{ij}\|_{C^{s}(\R^n)}.
\end{align}
The smoothing scale is chosen so that, after rescaling to a cube of sidelength $R:=\lambda^{-\frac{2-s}{2+s}}$, the regularity of the smoothed coefficients and the size of the errors behave as in the $C^{1,1}$ case, with respect to the new frequency $\mu:=\lambda R=\lambda^{\frac{2s}{2+s}}$. The following lemma proves~\eqref{small_ball_estimates-Lip}.

\begin{lemma}\label{lemma rescaling}
Let $n\geq 2$ and $s\in [1,2)$. Then
\begin{align}
\normLp{\rho_{\gamma}}{q_n/2}{(Q_R)}\lesssim\lambda^{2\delta(q_n)}R^{-1}(\normSch{\gamma}{\alpha(q_n)}{(L^2(Q_R^*))}+R^2\lambda^{-2}\normSch{(P-\lambda^2)\gamma (P-\lambda^2)}{\alpha(q_n)}{(L^2(Q_R^*))})
\end{align}
for any cube $Q_R$ of sidelength $R:=\lambda^{-\frac{2-s}{2+s}}$ and any finite-rank operator $\gamma\geq 0$ with $\Ran\gamma\subset H^2(\R^n)$. Here, $Q_R^*$ is the double of $Q_R$.
\end{lemma}

\begin{remark}
    By slight abuse of notation, we interpret, for example, the first term in the above inequality as
    \begin{align}
\normSch{\gamma}{\alpha(q_n)}{(L^2(Q_R^*))}:=\normSch{\mathbf{1}_{Q^*}\gamma\mathbf{1}_{Q^*}}{\alpha(q_n)}{(L^2(\R^n))},
    \end{align}
and we use analogous notation throughout the remainder of the article.  
\end{remark}

\begin{proof}
By Remark \ref{remark divergence versus non-divergence form}, we are free to replace $P=-\partial_i\g^{ij}\partial_j$ by $-\g^{ij}\partial_i\partial_j$.
We rescale via the unitary transformation
\begin{align}
U_R:L^2(\R^n)\to L^2(\R^n),\quad    U_Rf(x):=R^{n/2}f(Rx).
\end{align}
Note that $U_R:L^2(Q_R)\to L^2(Q)$, where $Q$ is a unit cube, and similarly for $Q^*,Q_R^*$.
If we set $\g_{R,s}(x):=\g_s(Rx)$, then
\begin{align}\label{def. P_R}
P_R:=-R^2 U_R(\g_s^{ij}\partial_i\partial_j)U_R^*=-\g_{R,s}^{ij}\partial_i\partial_j.
\end{align}
In view of \eqref{eq:smoothing bounds 2}, the coefficients of $\g_{R,s}$ have uniformly bounded $C^2$ norms,
\begin{align}
    |D^2\g^{ij}_{R,s}(x)|\lesssim R^2\lambda^{(2-s)\frac{2}{2+s}}=1.
\end{align}
Moreover, in terms of the new frequency $\mu=\lambda R$, the right-hand side of \eqref{eq:smoothing bounds 2} is of size $\mu^{-1}$, while the coefficients of $\g_{R,s}$ are truncated to frequencies of size at most $\mu^{1/2}$. Applying Proposition \ref{prop. C11 qm bound Rn} to $P_R$ and $\gamma_R:=U_R\gamma U_R^*$ at frequency $\mu$, we obtain
\begin{align}
 &\|\rho_{\gamma_R}\|_{L^{q_n/2}(Q)}
 \lesssim \mu^{2\delta(q_n)}
 (\normSch{\gamma_R}{\alpha(q_n)}{(L^2(Q^*))}+\mu^{-2}\normSch{(P_R-\mu^2)\gamma_R(P_R-\mu^2)}{\alpha(q_n)}{(L^2(Q^*))})\\
 &\qquad
 =\mu^{2\delta(q_n)}
 (\normSch{\gamma}{\alpha(q_n)}{(L^2(Q_R^*))}+\mu^{-2}R^4\normSch{(-\g_s^{ij}\partial_i\partial_j-\lambda^2)\gamma(-\g_s^{ij}\partial_i\partial_j-\lambda^2)}{\alpha(q_n)}{(L^2(Q_R^*))}),
\end{align}
where we used the unitary invariance of Schatten norms and \eqref{def. P_R} in the second line. Using the spectral decomposition of $\gamma$, one observes that
\begin{align}
\normLp{\rho_{\gamma_R}}{q_n/2}{(Q)}=R^{n(1-\frac{2}{q_n})}\normLp{\rho_{\gamma}}{q_n/2}{(Q_R)}=R^{2\delta(q_n)+1}\normLp{\rho_{\gamma}}{q_n/2}{(Q_R)}.
\end{align}
The result follows for the regularized operator $-\g_s^{ij}\partial_i\partial_j$. The same bound for the unregularized operator $P=-\g^{ij}\partial_i\partial_j$ follows perturbatively since
\begin{align}
 &\mu^{-2}R^4\normSch{(\g^{ij}-\g^{ij}_{s})\partial_i\partial_j\sqrt{\gamma}}{2\alpha(q_n)}{(L^2(Q_R^*))}^2\lesssim \mu^{-4}R^4 
 \normSch{(P+\lambda^2)\sqrt{\gamma}}{2\alpha(q_n)}{(L^2(Q_R^*))}^2\\
 &\lesssim \lambda^{-4} (\lambda^4\normSch{\sqrt{\gamma}}{2\alpha(q_n)}{(L^2(Q_R^*))}^2+\normSch{(P-\lambda^2)\sqrt{\gamma}}{2\alpha(q_n)}{(L^2(Q_R^*))}^2)\\
 &\lesssim \normSch{\sqrt{\gamma}}{2\alpha(q_n)}{(L^2(Q_R^*))}^2+R^2\lambda^{-2}\normSch{(P-\lambda^2)\sqrt{\gamma}}{2\alpha(q_n)}{(L^2(Q_R^*))}^2,
\end{align}
where we used \eqref{eq:smoothing bounds 2}, H\"older's inequality and elliptic regularity in the first inequality, and the fact that $\lambda^{-2}\leq R^2$ in the last inequality.
\end{proof}

\begin{remark}\label{lemma Lip cube sharp}
Since the estimate in Lemma \ref{lemma rescaling} is simply a rescaled version of \eqref{local bound qm Rn} in Proposition \ref{prop. C11 qm bound Rn}, we expect the Schatten exponent $\alpha(q_n)$ to be optimal.    
\end{remark}

\subsection{Energy estimate}\label{subsec:energy-estimates}
In the next step, we regularize $\g$ at the scale $\lambda$. In this subsection, we denote $\g_{\lambda}: = \varphi(\lambda^{-1} D) \g$ and $P_{\lambda}:=-\g_{\lambda}^{ij}\partial_i\partial_j$. Then
\begin{align}\label{smoothing to lambda scale}
	 \|\g^{ij}-\g^{ij}_{\lambda}\|_{L^{\infty}(\R^n)}&\lesssim \lambda^{-1}\|\g^{ij}\|_{C^1(\R^n)}.
\end{align}
Since \eqref{Lipschitz qm bound Rn} is stable under metric perturbations satisfying \eqref{smoothing to lambda scale}, we may replace $P$ by $P_{\lambda}$ there. This will allow us to use conical localization, as in Subsection \ref{subs. factorization}.

\begin{lemma}\label{lemma:energy_estimates_slabs-to-cube} 
Let $\eta\in C_c^{\infty}(\R^n)$ be supported in a conic region $\Gamma=\{|\xi'|\ll \xi_1\approx\lambda\}$.
Let $c\in \R$, $S_R:=\{x\in\R^n:|x_1-c|\leq R\}$ and $S_R^*:=S_{2R}$, where $R:=\lambda^{-\frac{2-s}{2+s}}$. Then, under the assumptions of Proposition \ref{prop. Lipschitz qm bound Rn},
\begin{align}
\normLp{\rho_{\eta(D)\gamma\eta(D)}}{q_n/2}{(S_R)}\lesssim\lambda^{2\delta(q_n)}\big(\normSch{\gamma}{\alpha(q_n)}{(L^2(\R^n))}+\lambda^{-2}\normSch{(P_{\lambda}-\lambda^2)\gamma (P_{\lambda}-\lambda^2)}{\alpha(q_n)}{(L^2(\R^n))}\big)
\end{align}
for any finite-rank operator $\gamma\geq 0$ with $\Ran\gamma\subset H^2(\R^n)$. The implicit constant is independent of $c$, $R$, $\lambda$ and $\gamma$.
\end{lemma}

\begin{proof}
Applying Lemma \ref{lemma rescaling} to $P_{\lambda}$, raising both sides to the power $q_n/2$ and summing over cubes $Q_R$ contained in $S_R$, we obtain
\begin{align}\label{mess}
\normLp{\rho_{\gamma}}{q_n/2}{(S_R)}&\lesssim\lambda^{2\delta(q_n)}R^{-1}\Bigg(\bigg(\sum_{Q_R\subset S_R}\normSch{\gamma}{\alpha(q_n)}{(L^2(Q_R^*))}^{q_n/2}\bigg)^{2/q_n}\\
&\qquad\qquad\qquad
+R^2\lambda^{-2}\left.\bigg(\sum_{Q_R\subset S_R}\normSch{(P_{\lambda}-\lambda^2)\gamma (P_{\lambda}-\lambda^2)}{\alpha(q_n)}{(L^2(Q_R^*))}^{q_n/2}\bigg)^{2/q_n}\right)
\\
&\lesssim
\lambda^{2\delta(q_n)}R^{-1}\Bigg(\bigg(\sum_{Q_R\subset S_R}\normSch{\gamma}{\alpha(q_n)}{(L^2(Q_R^*))}^{\alpha(q_n)}\bigg)^{1/\alpha(q_n)}\\
&\qquad\qquad\qquad
+R^2\lambda^{-2}\bigg(\sum_{Q_R\subset S_R}\normSch{(P_{\lambda}-\lambda^2)\gamma (P_{\lambda}-\lambda^2)}{\alpha(q_n)}{(L^2(Q_R^*))}^{\alpha(q_n)}\bigg)^{1/\alpha(q_n)}\Bigg)\\
&\leq\lambda^{2\delta(q_n)}R^{-1}\bigg(\normSch{\gamma}{\alpha(q_n)}{(L^2(S_R^*))}
+R^2\lambda^{-2}\normSch{(P_{\lambda}-\lambda^2)\gamma (P_{\lambda}-\lambda^2)}{\alpha(q_n)}{(L^2(S_R^*))}\bigg),
\end{align}
where we used that $\ell^{\alpha(q_n)}\subset \ell^{q_n/2}$, due to $q_n/2=(n+1)/(n-1)>(n+1)/n=\alpha(q_n)$, in the second inequality. The last inequality follows by interpolation between $\alpha=1$ and $\alpha=\infty$; the argument is the same as in the proof of \cite[Lemma 6.11]{Nguyen-2022}.

Using the commutator bound
\begin{align}
    \|[\g^{ij}_{\lambda},\eta(D)]\|\lesssim \lambda^{-1},
\end{align}
together with elliptic regularity, we obtain
\begin{align}
 \|[P_{\lambda},\eta(D)](P_{\lambda
 }+\lambda^2)^{-1}\|
 \lesssim
\lambda^{-1}\|D^2(P_{\lambda}+\lambda^2)^{-1}\|\lesssim\lambda^{-1}. 
\end{align}
This implies that
\begin{align}
\lambda^{-2}\|(P_{\lambda}-\lambda^2)\eta(D)\sqrt{\gamma}\|_{\mathfrak{S}^{2\alpha(q_n)}}^2
&\lesssim \lambda^{-2}\|\eta(D)(P_{\lambda}-\lambda^2)\sqrt{\gamma}\|_{\mathfrak{S}^{2\alpha(q_n)}}^2+\lambda^{-2}\|[P_{\lambda},\eta(D)]\sqrt{\gamma}\|_{\mathfrak{S}^{2\alpha(q_n)}}^2\\
&\lesssim \lambda^{-2}\|(P_{\lambda}-\lambda^2)\sqrt{\gamma}\|_{\mathfrak{S}^{2\alpha(q_n)}}^2+\lambda^{-4}\|(P_{\lambda}+\lambda^2)\sqrt{\gamma}\|_{\mathfrak{S}^{2\alpha(q_n)}}^2\\
&\lesssim \|\sqrt{\gamma}\|_{\mathfrak{S}^{2\alpha(q_n)}}^2+\lambda^{-2}\|(P_{\lambda}-\lambda^2)\sqrt{\gamma}\|_{\mathfrak{S}^{2\alpha(q_n)}}^2.
\end{align}
Therefore, applying \eqref{mess} to $\eta(D)\gamma\eta(D)$ yields
\begin{align}
\normLp{\rho_{\eta(D)\gamma\eta(D)}}{q_n/2}{(S_R)}\lesssim\lambda^{2\delta(q_n)}R^{-1}&\big(\normSch{\widetilde{\eta}(D)\gamma\widetilde{\eta}(D)}{\alpha(q_n)}{(L^2(S_R^*))}\\
& +R^2\lambda^{-2}\normSch{(P_{\lambda}-\lambda^2)\gamma (P_{\lambda}-\lambda^2)}{\alpha(q_n)}{(L^2(S_R^*))}\big),
\end{align}
where $\widetilde{\eta}$ is supported in $\Gamma$ and equal to $1$ on the support of $\eta$.
It remains to prove 
\begin{align}\label{energy estimate Lip.}
    R^{-1}\normSch{\widetilde{\eta}(D)\gamma\widetilde{\eta}(D)}{\alpha(q_n)}{(L^2(S_R^*))}\lesssim& \normSch{\gamma}{\alpha(q_n)}{(L^2(\R^n))}\\
    &+\lambda^{-2}\normSch{(P_{\lambda}-\lambda^2)\gamma (P_{\lambda}-\lambda^2)}{\alpha(q_n)}{(L^2(\R^n))}.
\end{align}
This is similar to \cite[Eq. (15)]{Smith-2006sharp}, but we will give a different proof in the next lemma. 
\end{proof}

By the same arguments as in Subsection \ref{factorization}, \eqref{energy estimate Lip.} follows from the following energy estimate for a first-order operator.

\begin{lemma}
Assume that $a(t,x',\xi',\lambda)\in \lambda S_{\lambda,\lambda}$. Then
\begin{align}
R^{-1}\normSch{\gamma}{\alpha(q_n)}{(L^2(S_R^*))}\lesssim \normSch{\gamma}{\alpha(q_n)}{(L^2(\R^n))}+\normSch{(D_t-a^w)\gamma (D_t-a^w)}{\alpha(q_n)}{(L^2(\R^n))}.
\end{align}
\end{lemma}

\begin{proof}
Since we have
\begin{align}
&R^{-1}\normSch{\gamma}{\alpha(q_n)}{(L^2(S_R^*))}
=R^{-1}\normSch{\mathbf{1}_{S_R^*}\sqrt{\gamma}}{2\alpha(q_n)}{(L^2(\R^n))}^2\\
&\leq R^{-1}\|\mathbf{1}_{S_R^*}(D_t-a^w-\I)^{-1}\|^2
\|(D_t-a^w-\I)\sqrt{\gamma}\|_{\mathfrak{S}^{2\alpha(q_n)}{(L^2(\R^n))}}^2\\ 
&\lesssim R^{-1}\|\mathbf{1}_{S_R^*}(D_t-a^w-\I)^{-1}\|^2(\normSch{\gamma}{\alpha(q_n)}{(L^2(\R^n))}+\normSch{(D_t-a^w)\gamma (D_t-a^w)}{\alpha(q_n)}{(L^2(\R^n))}),
\end{align}
the claim will follow from the resolvent estimate
\begin{align}\label{energy estimate first order}
 \|\mathbf{1}_{S_R^*}(D_t-a^w-\I)^{-1}\|^2\lesssim R.  
\end{align}
This in turn will follow from
\begin{align}\label{energy estimate first order 2}
 \|\mathbf{1}_{S_R^*}(D_t-a^w\mp\I)^{-1}\mathbf{1}_{S_R^*}\|\lesssim R  
\end{align}
due to the resolvent identity, which yields 
\begin{align}
  \|\mathbf{1}_{S_R^*}(D_t-a^w-\I)^{-1}\|^2 
  &= \|\mathbf{1}_{S_R^*}(D_t-a^w-\I)^{-1}(D_t-a^w+\I)^{-1}\mathbf{1}_{S_R^*}\|\\
  &\lesssim\sum_{\pm}\|\mathbf{1}_{S_R^*}(D_t-a^w\mp\I)^{-1}\mathbf{1}_{S_R^*}\|.
\end{align}
Using \eqref{variation of parameters}, \eqref{variation of parameters 2}, we estimate
\begin{align}
\|\mathbf{1}_{S_R^*}(D_t-a^w\mp\I)^{-1}\mathbf{1}_{S_R^*}f\|_{L^2(\R^{n-1})}\leq    \int_{-2R}^{2R}\|S(t,r)f(r)\|_{L^2(\R^{n-1})}\rd r\lesssim  \sqrt{R}\|f\|_{L^2(\R^n)},
\end{align}
and then
\begin{align}
\|\mathbf{1}_{S_R^*}(D_t-a^w\mp\I)^{-1}\mathbf{1}_{S_R^*}f\|_{L^2(\R^{n})}\lesssim\sqrt{R}\|f\|_{L^2(\R^n)} \Big(\int_{-2R}^{2R}\rd t\Big)^{1/2}\lesssim R \|f\|_{L^2(\R^n)}.
\end{align}
This proves \eqref{energy estimate first order 2}. 
\end{proof}

\begin{proof}[Proof of Proposition \ref{prop. Lipschitz qm bound Rn}]
 By a finite partition of unity, Lemma \ref{lemma:energy_estimates_slabs-to-cube} yields    
 \begin{align}
\normLp{\rho_{\gamma}}{q_n/2}{(S_R)}\lesssim\lambda^{2\delta(q_n)}(\normSch{\gamma}{\alpha(q_n)}{(L^2(\R^n))}+\lambda^{-2}\normSch{(P_{\lambda}-\lambda^2)\gamma (P_{\lambda}-\lambda^2)}{\alpha(q_n)}{(L^2(\R^n))}).
\end{align}
Raising this to the power $q_n/2$ and summing over $R^{-1}$ many slabs $S_R$ contained in $Q$ yields the claimed bound \eqref{Lipschitz qm bound Rn}.
\end{proof}

\begin{remark}\label{remark loss in the summation argument}
The second and the last inequality in \eqref{mess} are probably not sharp when $\tr \gamma$ is close to maximal. A more refined argument might lead to improved (larger) Schatten exponents.
\end{remark}

\section{The case $s\in [0,1)$ (proof of Theorem \ref{thm:main-result_Holder})}\label{sec:proof-Holder-case}

In this section, we complete the proof of Theorem \ref{thm:main-result_Holder} by addressing the H\"older case $s\in[0,1)$. Here, the scale $T:=\lambda^{-(1-s)}$ enters. On this scale, $C^s$ coefficients are well approximated by Lipschitz functions.
In this section, we consider operators in divergence form\footnote{The distinction between divergence form and non-divergence form operators matters for $s<1$. As pointed out in the previous section, the reductions for the Lipschitz case also follow from the form methods employed here.}. In other words, $P$ is defined via its quadratic form (see Subsection \ref{subsec. quad. form}). 
The commutator and smoothing errors are no longer relatively bounded in the operator sense, but we will show that they are relatively form-bounded. It will be more convenient to work with resolvents instead of quasimodes throughout this section.

\subsection{Spatial localization}
Given that we are now working with quadratic forms, we first revisit the reductions in Subsection \ref{sect. Reduction to Euclidean space}. 

\begin{lemma}
Let $\phi\in C^{\infty}(M)$. Then  
for $z\in\C$, $\re z\geq 1$, we have
\begin{align}\label{Holder commutator bound}
    \|(P+\re z)^{-1/2}[P,\phi](P-z)^{-1/2}\|\lesssim\frac{1}{|\im z|^{1/2}}.
\end{align}
\end{lemma}

\begin{proof}
Since $[P,\phi]u=-\nabla_{\g}\phi\cdot \nabla_{\g} u-\nabla\mathrlap{^*}_{\textup\g}(u\nabla_{\g}\phi)$, integration by parts yields
\begin{align}
 |\langle u,[P,\phi]u\rangle|
 \leq 2\|\nabla_{\g}\phi\|_{L^{\infty}(M)}\|u\|_{L^{2}(M)}    \|\nabla_{\g}u\|_{L^{2}(M)} 
 \lesssim \|u\|_{L^{2}(M)}Q(u)^{1/2}    
\end{align}
Since $Q(u)=\|P^{1/2}u\|^2$ and $\I[P,\phi]$ is  self-adjoint, this implies 
\begin{align}\label{commutator of P,phi sandwiched between sqrt(P+re z)}
 \|(P+\re z)^{-1/2}[P,\phi](P+\re z)^{-1/2}\|
 &\lesssim\|(P+\re z)^{-1/2}\|\|P^{1/2}(P+\re z)^{-1/2}\|
 \\&\lesssim (\re z)^{-1/2}.
\end{align}
By the spectral theorem,
\begin{align}\label{(P+1)^{1/2}(P-z)^{-1/2}}
\|(P+\re z)^{1/2}(P-z)^{-1/2}\|=\sup_{\tau^2\in\spec(P)}\frac{(\tau^2+\re z)^{1/2}}{|\tau^2-z|^{1/2}}\lesssim \frac{(\re z) ^{1/2}}{|\im z|^{1/2}},    
\end{align}
hence 
\begin{align}
&\|(P+\re z)^{-1/2}[P,Q](P-z)^{-1/2}\|\\
&\leq \|(P+\re z)^{-1/2}[P,Q](P+\re z)^{-1/2}\|\|(P+\re z)^{1/2}(P-z)^{-1/2}\|\lesssim\frac{1}{|\im z|^{1/2}}. 
\end{align}
\end{proof}

By a partition of unity argument, we may assume $W$ is supported in a fixed chart $\Omega\subset M$. Let $\kappa:\Omega\to\kappa(\Omega)\subset\R^n$ be a diffeomorphism. For clarity of exposition, we temporarily denote the operator $P$ in local coordinates by $\tilde{P}:=(\kappa^*)^{-1}P\kappa^*$, and similarly $\tilde{W}:=(\kappa^*)^{-1}W\kappa^*$, where
\begin{align}
\kappa^*:C^{\infty}(\kappa(\Omega))\to C^{\infty}(\Omega),
\quad \kappa^*f(x):=f(\kappa(x)),
\end{align}
is the pullback operator, and $(\kappa^*)^{-1}=(\kappa^{-1})^*$. These operators extend to isometries on the corresponding $L^2$ spaces.
As before, we extend the coefficients $\g^{ij}$ to all of $\R^n$ in such a way that $\g^{ij}=\delta^{ij}$ outside a ball, and we consider $\tilde{P}$ as an operator on $L^2(\R^n,|\g|^{1/2}\rd x)$. In contrast to Subsection \ref{sect. Reduction to Euclidean space}, we refrain from changing the inner product to the standard one since \eqref{change of Hilbert space} would no longer be of divergence form. However, we note that Schatten norms over either $L^2(\R^n,|\g|^{1/2}\rd x)$ or $L^2(\R^n)$ are equivalent.

\begin{lemma}\label{lemma Holder manifold to Rn}
Assume that $W$ is supported in a chart $\Omega\subset M$. Then
\begin{align}
    \|W(P-z)^{-1}\|_{\mathfrak{S}^{2\alpha(q_n)'}(L^2(M))}\lesssim&
    \|\tilde{W}(\tilde{P}-z)^{-1}\|_{\mathfrak{S}^{2\alpha(q_n)'}(L^2(\R^n))}\\
    &+|\im z|^{-1}\|\tilde{W}(\tilde{P}-z)^{-1}(\tilde{P}+\re z)^{1/2}\|_{\mathfrak{S}^{2\alpha(q_n)'}(L^2(\R^n))}.
\end{align}
\end{lemma}

\begin{proof}
Let $\phi$ be a bump function supported in $\Omega$.
A straightforward calculation shows that
\begin{align}
    \phi(P-z)^{-1}=\kappa^*(\tilde{P}-z)^{-1}(\kappa^*)^{-1}\phi
    +\kappa^*(\tilde{P}-z)^{-1}(\kappa^*)^{-1}[P,\phi](P-z)^{-1}
\end{align}
as bounded operators on $L^2(M)$.
Thus, 
\begin{align}
\|W&(P-z)^{-1}\|_{\mathfrak{S}^{2\alpha(q_n)'}}
\\&\leq \|W\kappa^*(\tilde{P}-z)^{-1}(\kappa^*)^{-1}\phi\|_{\mathfrak{S}^{2\alpha(q_n)'}}
+\|W\kappa^*(\tilde{P}-z)^{-1}(\kappa^*)^{-1}[P,\phi](P-z)^{-1}\phi\|_{\mathfrak{S}^{2\alpha(q_n)'}}\\
&\lesssim \|\tilde{W}(\tilde{P}-z)^{-1}\|_{\mathfrak{S}^{2\alpha(q_n)'}}\\
&\quad
+\|\tilde{W}(\tilde{P}-z)^{-1}(\kappa^*)^{-1}(P+\re z)^{1/2}\|_{\mathfrak{S}^{2\alpha(q_n)'}}
\|(P+\re z)^{-1/2}[P,\phi](P-z)^{-1}\phi\|.
\end{align}
Further, we have
\begin{align}
\|\tilde{W}&(\tilde{P}-z)^{-1}(\kappa^*)^{-1}(P-z)^{1/2}\|_{\mathfrak{S}^{2\alpha(q_n)'}}\\
&\leq \|\tilde{W}(\tilde{P}-z)^{-1/2}\|_{\mathfrak{S}^{2\alpha(q_n)'}}
\|(\tilde{P}-z)^{-1/2}(\kappa^*)^{-1}(P-z)^{1/2}\|\\
&\lesssim \|\tilde{W}(\tilde{P}-z)^{-1/2}\|_{\mathfrak{S}^{2\alpha(q_n)'}},
\end{align}
and, using \eqref{Holder commutator bound}, 
\begin{align}
\|(P+\re z)^{-1/2}[P,\phi](P-z)^{-1}\phi\|
&\lesssim\|(P+\re z)^{-1/2}[P,\phi](P-z)^{-1/2}\| \|(P-z)^{-1/2}\| \\
&\lesssim |\im z|^{-1}.
\end{align}
This proves the claim.
\end{proof}

We drop the tilde and denote the operator associated to the quadratic form
 \begin{align}\label{def. quad. form Holder}
Q(u):=\int_{\R^n}\g^{ij}\partial_i\overline{u}\partial_j u\,|\g|^{1/2}\rd x,\quad u\in H^1(\R^n,|\g|^{1/2}\rd x),
\end{align}
again by $P$. 
By Proposition \ref{prop. equivalence of cluster, qm and resolvent} and Lemma \ref{lemma Holder manifold to Rn}, with $z=(\lambda\pm\I\lambda^{1-s})^2$, it remains to prove the following proposition.

\begin{proposition}
Let $P$ be the self-adjoint operator on $L^2(\R^n,|\g|^{1/2}\rd x)$ associated to the quadratic form \eqref{def. quad. form Holder}. Then
\begin{align}
 \normSch{W(P-(\lambda\pm\I\lambda^{1-s})^2)^{-1}}{2\alpha(q_n)'}{(L^2(\R^n))}^2
 &\lesssim\ \lambda^{2\delta_s(q_n)-2} \|W\|_{L^{2(q_n/2)'(\R^n)}}^2, \label{Holder Prop. bound 1}  \\
  \normSch{W(P-(\lambda\pm\I\lambda^{1-s})^2)^{-1}(P+\lambda^2)^{1/2}}{2\alpha(q_n)'}{(L^2(\R^n))}^2
 &\lesssim\ \lambda^{2\delta_s(q_n)} \|W\|_{L^{2(q_n/2)'}(\R^n)}, \label{Holder Prop. bound 2}
\end{align}
for all $W\in L^{2(q_n/2)'}(\R^n)$ supported in a unit ball.
\end{proposition}

By Proposition \ref{prop:elliptic} and Remark \ref{remark KSS valid for Rn}, we have the following elliptic estimates, 
\begin{align}
 \normSch{W\Pi_{>2\lambda}(P-(\lambda\pm\I\lambda^{1-s})^2)^{-1}}{2(q_n/2)'}{}^2
 &\lesssim\ \lambda^{2s\delta(q_n)-2} \|W\|_{L^{2(q_n/2)'}}^2,   \\
  \normSch{W\Pi_{>2\lambda}(P-(\lambda\pm\I\lambda^{1-s})^2)^{-1}(P+\lambda^2)^{1/2}}{2(q_n/2)'}{}^2
 &\lesssim\ \lambda^{2s\delta(q_n)} \|W\|_{L^{2(q_n/2)'}},   
\end{align}
for any $W\in L^{2(q_n/2)'}(\R^n)$, which is better than needed.

\subsection{Microlocalization, regularization and rescaling}
The spectrally localized versions of \eqref{Holder Prop. bound 1}, \eqref{Holder Prop. bound 2} would follow from the bound  
\begin{align}
    \normSch{W\Pi_{\leq 2\lambda}(P-(\lambda\pm\I\lambda^{1-s})^2)^{-1}}{2\alpha(q_n)'}{}^2
 &\lesssim\ \lambda^{2\delta_s(q_n)-2} \|W\|_{L^{2(q_n/2)'}}^2.
\end{align}
By Lemma~\ref{lemma insert req. loc.} and Remark \ref{remark lemma insert req. loc.}, we may replace $\Pi_{\leq 2\lambda}$ by a Fourier multiplier $\varphi(c\lambda^{-1}D)$. Then, by a conic partition of unity, we may replace the latter by a multiplier $\eta(D)$ as in Lemma \ref{lemma:energy_estimates_slabs-to-cube}.   

As in Subsection \ref{subsec:energy-estimates}, we set the regularized metrics $\g_{\lambda}:=\varphi(\lambda^{-1}D)\g$. Similarly to \eqref{smoothing to lambda scale}, we have
\begin{align}
	 \|\g^{ij}-\g^{ij}_{\lambda}\|_{L^{\infty}(\R^n)}&\lesssim \lambda^{-s}\|\g^{ij}\|_{C^s(\R^n)}.
\end{align}
We denote the corresponding regularized operator again by $P_{\lambda}$.

\begin{lemma}\label{lemma form boundedness P-Plambda Holder}
The perturbation $P-P_{\lambda}$ is relatively form bounded with respect to $P$, and for $z\in\C$, $\re z\geq 1$, we have 
\begin{align}
    \|(P-z)^{-1/2}(P-P_{\lambda})(P-z)^{-1/2}\|&\lesssim \lambda^{-s}\frac{\re z}{|\im z|},\label{Holder smoothing bound}\\
    \|(P+\re z)^{-1/2}(P-P_{\lambda})(P+\re z)^{-1/2}\|&\lesssim \lambda^{-s}. \label{Holder smoothing bound 2}
\end{align}
\end{lemma}

\begin{proof}
 The quadratic form associated to $P$ is given by \eqref{def. quad. form Holder},
and similarly we define the quadratic form $Q_{\lambda}$ of $P_{\lambda}$, with $\g_{\lambda}$ in place of $\g$. Then
\begin{align}
    |Q(u)-Q_{\lambda}(u)|
    &\leq \int_{\R^d} \abs{\g^{ij}-\g_{\lambda}^{ij}\frac{|\g_{\lambda}|^{1/2}}{|\g|^{1/2}}}| \partial_i u| |\partial_j u| |\g|^{1/2}\rd x\\
    &\lesssim \Big(\|\g-\g_{\lambda}\|_{L^{\infty}}+\|\g_{\lambda}\|_{L^{\infty}}\Big\|1-\frac{|\g_{\lambda}|^{1/2}}{|\g|^{1/2}}\Big\|_{L^{\infty}}\Big)\|\nabla u\|_{L^2}^2\\
    &\lesssim \lambda^{-s}Q(u).
\end{align}
By \eqref{(P+1)^{1/2}(P-z)^{-1/2}}, this implies
\begin{align}
\|(P-z)^{-1/2}(P-P_{\lambda})(P-z)^{-1/2}\|\lesssim \lambda^{-s} \|P^{1/2}(P-z)^{-1/2}\|^2\lesssim \lambda^{-s}\frac{\re z}{|\im z|}.
\end{align}
The second inequality is proved similarly.
\end{proof}

\begin{remark}\label{remark form boundedness P-Plambda Holder}
    It is clear that the roles of $P$ and $P_{\lambda}$ in Lemma \ref{lemma form boundedness P-Plambda Holder} can be interchanged.
\end{remark}

We note the following resolvent identity
\begin{align}\label{Tiktopoulos' Formula}
    (P_{\lambda}-z)^{-1}=(P-z)^{-1/2}(\mathbf{1}+(P-z)^{-1/2}(P_{\lambda}-P)(P-z)^{-1/2})^{-1}(P-z)^{-1/2},
\end{align}
see, e.g., \cite[Lemma 1]{MR1308388}, \cite[Chap.VI.3.1.]{MR407617}, \cite[Chapter II.3]{MR455975}. 
Lemma \ref{lemma form boundedness P-Plambda Holder} shows that, for $z=(\lambda\pm\I C\lambda^{1-s})^2$, 
\begin{align}
    \|(P-z)^{-1/2}(P_{\lambda}-P)(P-z)^{-1/2}\|\lesssim C^{-1}.
\end{align}
Choosing $C$ sufficiently large, we may thus assume that the above norm is $<1$. 
Increasing the imaginary part of $z$ by a multiplicative constant does not affect the argument, and we drop $C$ from the notation. 
By \eqref{Tiktopoulos' Formula}, Lemma \ref{lemma form boundedness P-Plambda Holder} and Remark \ref{remark form boundedness P-Plambda Holder}, 
\begin{align}
&\normSch{W\eta(D)((P-(\lambda\pm\I\lambda^{1-s})^2)^{-1}-(P_{\lambda}-(\lambda\pm\I\lambda^{1-s})^2)^{-1})}{2\alpha(q_n)'}{}^2\\
 &\quad\lesssim \normSch{W\eta(D)(P_{\lambda}-(\lambda\pm\I\lambda^{1-s})^2)^{-1}}{2\alpha(q_n)'}{}^2.
\end{align}
Hence it suffices to prove 
\begin{align}
 \normSch{W\eta(D)(P_{\lambda}-(\lambda\pm\I\lambda^{1-s})^2)^{-1}}{2\alpha(q_n)'}{}^2\lesssim\lambda^{2\delta_s(q_n)-2}\|W\|_{L^{2(q_n/2)'}}^2   
\end{align}
for all $W\in L^{2(q_n/2)'}(\R^n)$ supported in a unit cube $Q$.
Using Proposition \ref{prop. equivalence of cluster, qm and resolvent} once more, we are reduced to proving the quasimode bound
\begin{align}\label{Holder final estimate to prove}
\|\rho_{\eta(D)\gamma\eta(D)}\|_{L^{q_n/2}(L^2(Q))}& \lesssim (R^{-1}T)^{\frac{2}{q_n}}\lambda^{2\delta(q_n)} T^{-1} \Big(\|\gamma\|_{\mathfrak S^{\alpha(q_n)}(L^2(\R^n))}\\
    &\quad+\lambda^{-2}T^2 \|(P_{\lambda}-\lambda^2)\gamma (P_{\lambda}-\lambda^2)\|_{\mathfrak S^{\alpha(q_n)}(L^2(\R^n))}\Big),
\end{align}
where we recall that $R=\lambda^{-\frac{2-s}{2+s}}$ and $T=\lambda^{-(1-s)}$.

\begin{proof}[Proof of \eqref{Holder final estimate to prove}]
We set $\g_{\lambda, T}(x):=\g_\lambda(Tx)$, $P_{\lambda, T}:=\partial_i \g_{\lambda, T} \partial_j= T^2 U_T P_\lambda U_{T}^*$, and $\gamma_{T}:=U_{T}^* \gamma U_{T}$. The new frequency is $\mu:=\lambda T=\lambda^s$. Then, the rescaled 
operator $P_{\lambda, T}$  behaves like an operator with uniform Lipschitz bounds in the sense that  
\begin{align*}
    \|D \g_{\lambda, T}^{ij}\|_{L^\infty} =T \|(D \g_{\lambda}^{ij})(T\cdot)\|_{L^\infty} \lesssim T\lambda^{1-s}  \|\g^{ij}\|_{C^s}= \|\g^{ij}\|_{C^s}.
\end{align*}
Applying Lemma \ref{lemma:energy_estimates_slabs-to-cube} to $P_{\lambda, T}$, with the sidelength $R$ replaced by 
$RT^{-1}$ and at frequency $\mu$, we get 
\begin{align*}
\|\rho(\eta(D)\gamma_{T}\eta(D))\|_{L^{q_n/2}(S_{RT^{-1}})} \lesssim \mu^{2\delta(q_n)} &\Big(\|\gamma_{T}\|_{\mathfrak S^{\alpha(q_n)}(L^2(Q^*))}
    \\&+\mu^{-2} \|(P_{\lambda, T}-\mu^2)\gamma_{T} (P_{\lambda, T}-\mu^2)\|_{\mathfrak S^{\alpha(q_n)}(L^2(Q^*))}\Big).
\end{align*}
Then, summing over all slabs $S_{R{T^{-1}}}$ in $Q$ gives 
\begin{align*}
    \|\rho(\eta(D)\gamma_{T}\eta(D))\|_{L^{q_n/2}(Q)} 
    \lesssim (R^{-1}T)^{\frac{2}{q_n}}\mu^{2\delta(q_n)} &\Big(\|\gamma_{T}\|_{\mathfrak S^{\alpha(q_n)}(L^2(Q^*)}
    \\&+\mu^{-2} \|(P_{\lambda, T}-\mu^2)\gamma_{T} (P_{\lambda, T}-\mu^2)\|_{\mathfrak S^{\alpha(q_n)}(L^2(Q^*))}\Big).
\end{align*}
By scaling, it follows that
    \begin{align*}
\|\rho_{\eta(D)\gamma\eta(D)}\|_{L^{q_n/2}(Q_T)} 
    \lesssim& (R^{-1}T)^{\frac{2}{q_n}}\lambda^{2\delta(q_n)} T^{-1}\times
    \\&\times\Big(\|\gamma\|_{\mathfrak S^{\alpha(q_n)}(L^2(Q_T^*))}
    +\lambda^{-2}T^2 \|(P_{\lambda}-\lambda^2)\gamma (P_{\lambda}-\lambda^2)\|_{\mathfrak S^{\alpha(q_n)}(L^2(Q_{T}^*))}\Big),
\end{align*}
where $Q_T$ is a cube of sidelength $T$.
Finally, summing over cubes $Q_T$ contained in a unit cube $Q$ yields (using the same summation argument as in \eqref{mess})
\begin{align*}
\|\rho_{\eta(D)\gamma\eta(D)}\|_{L^{q_n/2}} \lesssim (R^{-1}T)^{\frac{2}{q_n}}\lambda^{2\delta(q_n)} T^{-1} &\Big(\|\gamma\|_{\mathfrak S^{\alpha(q_n)}(L^2(\R^n))}\\&+\lambda^{-2}T^2 \|(P_{\lambda}-\lambda^2)\gamma (P_{\lambda}-\lambda^2)\|_{\mathfrak S^{\alpha(q_n)}(L^2(\R^n))}\Big).
\end{align*}
This proves \eqref{Holder final estimate to prove}.
\end{proof}

\subsection*{Acknowledgements and Declarations}

J.-C.~C. and X.~S. acknowledge support through the Engineering \& Physical Sciences Research Council (EP/X011488/1).
 
J.-C.~C. and N.N.~N.\ acknowledge partial financial support from the London Mathematical Society, research in pairs grant nr.~42341.

N.N.~N.\ has been supported from INI and the EPSRC grant (nr. ~EP/V521929/1) via the UK Spectral Theory Network, and from the European Union through the European Research Council's Starting Grant \textsc{FermiMath}, grant agreement nr.~101040991. Views and opinions expressed are those of the authors and do not necessarily reflect those of the European Union or the European Research Council Executive Agency. Neither the European Union nor the granting authority can be held responsible for them.


\bibliographystyle{siam}
\bibliography{biblio.bib}

\end{document}